\documentclass[10pt,a4paper]{amsart}
\usepackage[a4paper,marginratio=1:1]{geometry}
\usepackage[protrusion=all]{microtype}
\usepackage{amssymb}
\usepackage[initials,non-sorted-cites]{amsrefs}
\usepackage{mybibspec}
\usepackage[shortlabels]{enumitem}
\usepackage[symbol,perpage]{footmisc}
\usepackage{graphicx}
\usepackage{braket,tensor,bm,stmaryrd,fouridx}

\numberwithin{equation}{section}
\usepackage{mathtools}\mathtoolsset{showonlyrefs}

\newcommand{\abs}[1]{\lvert#1\rvert}

\newcommand{\transpose}[1]{\fourIdx{t}{}{}{}{#1}}
\newcommand{\bdry}{\partial}
\newcommand{\Eren}{\mathcal{E}_\mathrm{ren}}
\newcommand{\Aren}{\mathcal{A}_\mathrm{ren}}
\newcommand{\Vren}{\mathcal{V}_\mathrm{ren}}
\DeclareMathOperator{\Ric}{Ric}
\DeclareMathOperator{\tr}{tr}

\DeclareMathOperator{\grad}{grad}
\newcommand{\PO}{\mathit{PO}}
\newcommand{\PGL}{\mathit{PGL}}
\newcommand{\AdSplane}{\mathrm{AdS}_2}

\newtheorem{thm}{Theorem}[section]
\newtheorem{prop}[thm]{Proposition}
\newtheorem{lem}[thm]{Lemma}

\theoremstyle{definition}
\newtheorem{dfn}[thm]{Definition}
\theoremstyle{remark}
\newtheorem{rem}[thm]{Remark}

\title[Harmonic maps and conformal geodesics]{Asymptotic Dirichlet problem for harmonic maps and\\ conformal geodesics}
\author{Yoshihiko Matsumoto}
\address{Department of Mathematics, Graduate School of Science, Osaka University, Toyonaka, Osaka 560-0043, Japan}
\email{matsumoto.yoshihiko.sci@osaka-u.ac.jp}
\subjclass[2020]{Primary 53B20; Secondary 53B30, 53C18, 53C43, 58J32.}

\begin{document}

\maketitle

\begin{abstract}
	The asymptotic Dirichlet problem for harmonic maps from the hyperbolic plane into
	conformally compact Einstein manifolds is
	used to give a holographic characterization of conformal geodesics on the boundary at infinity,
	in a way deeply inspired by a work of Fine and Herfray on renormalized area minimization.
\end{abstract}

\section*{Introduction}

In this paper, we give the following holographic characterization of conformal geodesics,
which is a variant of a theorem of Fine and Herfray \cite{Fine-Herfray-22}.
We identify the hyperbolic plane $\mathbb{H}^2$ with the Poincar\'e disk
and call its boundary the boundary circle at infinity $\bdry_\infty\mathbb{H}^2$.
Moreover, we write $r$ for the hyperbolic distance from any fixed point in $\mathbb{H}^2$.

\begin{thm}
	\label{thm:main}
	Let $(X,g_+)$ be a conformally compact Einstein manifold of dimension $n+1\geqq 3$
	with smooth conformal infinity $\bdry_\infty X$.
	Suppose $\gamma\colon\bdry_\infty\mathbb{H}^2\to\bdry_\infty X$ is a smooth map
	having nowhere vanishing differential.

	(1)
	There is always a formal polyhomogeneous proper map $u\colon\mathbb{H}^2\to X$
	with boundary value $\gamma$ that is asymptotically harmonic to infinite order, i.e.,
	\begin{equation}
		\label{eq:harmonicity-to-infinite-order}
		\abs{\tau(u)}=o(e^{-m r})\qquad
		\text{as $r\to\infty$ for any $m>0$}.
	\end{equation}
	Also, such a formal extension $u$ of $\gamma$ always satisfies $\abs{\nabla du}=O(e^{-2r})$.

	(2)
	There exists such a formal polyhomogeneous extension $u\colon\mathbb{H}^2\to X$ of $\gamma$ satisfying
	\begin{equation}
		\label{eq:isometric-totally-geodesic-to-third-order}
		\abs{\nabla du}=o(e^{-2r})\qquad
		\text{as $r\to\infty$}
	\end{equation}
	if and only if $\gamma$ is a conformal geodesic in $\bdry_\infty X$.
\end{thm}

Let us clarify the statement of the theorem.

The second fundamental form $\nabla du$ of a map $u\colon Y\to X$ between Riemannian manifolds
in the sense of Eells--Sampson \cite{Eells-Sampson-64}
is the covariant derivative of $du\in\Gamma(T^*Y\otimes u^*TX)$.
In other words, $\nabla du$ is a section of $T^*Y\otimes T^*Y\otimes u^*TX$ given by the local formula
\begin{equation}
	\tensor{(\nabla du)}{_p_q^k}=\partial_p\tensor{(du)}{_q^k}
	-\tensor{\Gamma}{^r_p_q}\tensor{(du)}{_r^k}
	+(\tensor{\Gamma}{^k_i_j}\circ u)\tensor{(du)}{_p^i}\tensor{(du)}{_q^j},
\end{equation}
where $\tensor{\Gamma}{^r_p_q}$ and $\tensor{\Gamma}{^k_i_j}$ are the Christoffel symbols of $Y$ and $X$.
When $u$ is an isometric embedding, $\nabla du$ is nothing but the
usual second fundamental form of a submanifold.
Recall also the tension field $\tau(u)$ is the trace of $\nabla du$ with respect to the metric of $Y$,
and $u$ is called a harmonic map if and only if $\tau(u)$ vanishes.
In the statement of the theorem, the pointwise norms of $\tau(u)$ and $\nabla du$ are
measured with respect to the hyperbolic metric of $\mathbb{H}^2$ and $g_+$.

\emph{Conformal geodesics}, or more traditionally \emph{conformal circles}, are distinguished curves
$\gamma\colon I\to M$ (where $I$ is an interval) in a conformal manifold $(M,[g])$
characterized by the condition that for any $t\in I$ there exists
a neighborhood of $t$ in which $\gamma$ satisfies, for an arbitrarily fixed
representative metric $g\in[g]$ and for some 1-form $\alpha$ defined near $\gamma(t)$,
\begin{subequations}
	\label{eq:system-for-conformal-geodesics-in-terms-of-one-form}
\begin{gather}
	\label{eq:conformal-geodesics-in-terms-of-one-form-1}
	\nabla_{\dot{\gamma}}\dot{\gamma}+2\alpha(\dot{\gamma})\dot{\gamma}-\abs{\dot{\gamma}}^2\alpha^\sharp=0,\\
	\label{eq:conformal-geodesics-in-terms-of-one-form-2}
	\nabla_{\dot{\gamma}}\alpha^\sharp-P(\dot{\gamma},\cdot)^\sharp-\alpha(\dot{\gamma})\alpha^\sharp
	+\frac{1}{2}\abs{\alpha}^2\dot{\gamma}=0
\end{gather}
\end{subequations}
where $\nabla$ is the Levi-Civita connection,
$P=\frac{1}{n-2}(\Ric-\frac{1}{2(n-1)}Rg)$ is the Schouten tensor (where $n=\dim M$),
and $\sharp$ denotes the metric dual, all taken with respect to the representative metric $g$.
It turns out that this condition on a curve $\gamma$ is irrelevant to any particular choice of $g$.
(We have implicitly assumed that $n\geqq 3$ here,
for in $n=2$ the definition of the Schouten tensor does not make sense.
When $n=2$, a symmetric 2-tensor $P$ that behaves similarly to the Schouten tensor in higher dimensions
is taken as a part of the geometric structure on $M$, in which case $M$ carries a M\"obius structure
in the sense of Calderbank \cite{Calderbank-98}.
Then conformal geodesics, or M\"obius geodesics, in $M$ can be defined by the same formulae.)

An important fact regarding conformal geodesics is that they allow projective reparametrizations,
and this is why it makes sense to take the circle $\bdry_\infty\mathbb{H}^2$, or a part of it,
as the domain of a conformal geodesic $\gamma$
(see \S\ref{sec:conformal-geodesics} for details).
Then it is a natural idea from the viewpoint of ``holographic principle'' to consider extending $\gamma$
into some preferable mapping $u\colon\mathbb{H}^2\to X$,
the target space $X$ being a conformally compact Einstein filling of $M$.
Theorem \ref{thm:main} provides an understanding of conformal geodesics along this line.

By ``a formal polyhomogeneous proper map $u$,''
we mean that $u$ is actually a collection of formal expansions
\eqref{eq:polyhomogeneous-expansion-of-map}
given for any choice of an identification of $\mathbb{H}^2$ with the upper-half plane,
any choice of a Graham--Lee normalization \eqref{eq:Graham-Lee-normal-form} of the target space $(X,g_+)$,
and any choice of local coordinates in $\bdry_\infty X$.
In view of this, we could also think of $g_+$ as a formal asymptotic polyhomogenous expansion.
Theorem \ref{thm:main-theorem-in-action} will make it explicit.

As we have indicated, our harmonic map approach toward conformal geodesics
is strongly inspired by the work \cite{Fine-Herfray-22} of Fine--Herfray
using minimal surfaces in $(X,g_+)$ that are critical with respect to
the renormalized area $\mathcal{A}_\mathrm{ren}$.
Their result can be summarized as follows.

\begin{thm}[Fine--Herfray \cite{Fine-Herfray-22}]
	\label{thm:Fine-Herfray}
	Let $(X,g_+)$ be a conformally compact Einstein manifold with smooth conformal infinity.
	Suppose $\Sigma\subset X$ is an even polyhomogeneous formal surface
	with $\gamma(I)=\overline{\Sigma}\cap\bdry_\infty X$,
	where $\overline{\Sigma}$ is the closure of $\Sigma$
	in the conformal compactification $\overline{X}=X\cup\bdry_\infty X$ and
	$\gamma\colon I\to\bdry_\infty X$ is a curve with nowhere vanishing velocity,
	that is formally critical for the renormalized area $\Aren$.
	Additionally, let $t_0\in I$, and suppose that
	$\varphi\colon\overline{U}\to \overline{V}\subset\overline{\Sigma}$ is
	a homeomorphism between an open neighborhood $\overline{U}$ in
	the closed upper-half plane $\smash{\overline{\mathbb{R}}}^2_+=\set{(t,s)|s\geqq 0}$ of $(t_0,0)$
	and an open neighborhood $\overline{V}$ of $\gamma(t_0)$ in $\overline{\Sigma}$
	satisfying the following conditions:
	\begin{enumerate}[(i)]
		\item
			$\varphi$ has an even polyhomogeneous asymptotic expansion
			as a map from $\overline{U}$ into $\overline{X}$.
		\item
			$\varphi$ restricts to a diffeomorphism between $U$ and $V$,
			where $U=\overline{U}\setminus\bdry\smash{\overline{\mathbb{R}}}^2_+$
			and $V=\overline{V}\setminus\bdry\overline{X}$.
		\item
			\label{item:isothermality-in-FH-theorem}
			$\varphi\colon U\to V$ gives an isothermal (i.e., conformal) parametrization of $V\subset\Sigma$.
		\item
			Near $t_0\in I$,
			the parametrization $\gamma\colon I\to \bdry_\infty X$ of the boundary curve equals
			the restriction of $\varphi$ to $\overline{U}\cap\set{s=0}$.
	\end{enumerate}
	Simply put, $\varphi$ is an isothermal polyhomogenous local parametrization of $\Sigma$
	by $U\subset\mathbb{H}^2$, and we are writing $\gamma$ for its boundary value.
	Then the following holds.

	(1) Such a parametrized formal surface $\varphi\colon U\to V\subset\Sigma$ necessarily satisfies
	\begin{equation}
		\label{eq:isometric-totally-geodesic-to-second-order-Fine-Herfray}
		\varphi^*g_+=g_{\mathbb{H}^2}+O(s^2),\qquad
		\mathrm{II}=O(s^2)\qquad
		\text{as $s\to 0$},
	\end{equation}
	where $\mathrm{II}$ denotes the second fundamental form of the surface $\Sigma$.

	(2) Such a parametrized formal surface $\varphi\colon U\to V\subset\Sigma$ satisfies
	\begin{equation}
		\label{eq:isometric-totally-geodesic-to-third-order-Fine-Herfray}
		\varphi^*g_+=g_{\mathbb{H}^2}+o(s^2),\qquad
		\mathrm{II}=o(s^2)\qquad
		\text{as $s\to 0$},
	\end{equation}
	possibly in a smaller neighborhood of $(t_0,0)$,
	if and only if $\gamma$ is a conformal geodesic near $t_0\in I$.
\end{thm}

Theorem \ref{thm:Fine-Herfray} is philosophically
a holographic characterization of \emph{unparametrized} conformal geodesics
in terms of an $\Aren$-critical minimal surface $\Sigma$,
although in its statement one needs to take some isothermal coordinates in $V\subset\Sigma$,
which in turn gives a preferred parametrization of the boundary conformal geodesic.
The relationship between Theorems \ref{thm:Fine-Herfray} and \ref{thm:main} can be compared to
the length-minimization and the energy-minimization characterizations of
geodesics in Riemannian manifolds.

In fact, the asymptotic expansion of our harmonic map $u\colon\mathbb{H}^2\to X$ in Theorem \ref{thm:main}
is precisely recovered by the expansion of the $\Aren$-critical minimal isothermal immersion $\varphi\colon U\to X$
given in Theorem \ref{thm:Fine-Herfray},
which is mostly obvious because any isothermal parametrization of a minimal surface gives a harmonic mapping.
The virtue of Theorem \ref{thm:main} is that it identifies the map $u$
relying only upon harmonicity, a weaker condition than (the conjunction of) minimality and isothermality.
This refined understanding makes the necessary steps for the proof more streamlined,
and it also provides us a practical means to tackle the similar, but conceptually and computationally more intricate,
case of curves in CR manifolds, which will be explored in a separate paper.

We also point out that Theorem \ref{thm:main} does not use the renormalized energy in its statement,
while Theorem \ref{thm:Fine-Herfray} uses the renormalized area.
Actually, the notion of the renormalized energy $\Eren$ of a harmonic map $u$ can be defined quite naturally
and was used in the formulation of the main theorem in an earlier version of this manuscript.
But it was realized that the $\Eren$-criticality of $u$ automatically implied by
the asymptotic total geodesicness property in Theorem \ref{thm:main} (2).
Therefore, we removed the reference to the $\Eren$-criticality from the statement,
and we put the whole discussion of the renormalized energy in the appendix rather than the main text.

Lastly, we emphasize that in Theorem \ref{thm:main-theorem-in-action}
our characterization of conformal geodesics is extended
to the case of indefinite signature conformal classes on the boundary at infinity.
It is also clear that our approach has a room for higher-dimensional generalizations, which is not pursued here.

I would like to thank Rod Gover for informing me about the work \cite{Fine-Herfray-22} of Fine and Herfray.
I was benefited by discussions with Rafe Mazzeo, and I am thankful to Olivier Biquard for
letting me know B\'erard's work \cite{Berard-13}, which gives a notion of renormalized energy
different from the one we discuss in Appendix \ref{sec:renormalized-energy}.
This work was partially supported by JSPS KAKENHI Grant Number 20K03584 and 24K06738.

\section{Conformal geodesics}
\label{sec:conformal-geodesics}

The conformal invariance of conformal geodesics can be best understood through defining them
in terms of the normal Cartan connection.
For simplicity, we assume that $n\geqq 3$ throughout this section.

Recall that the $n$-dimensional standard pseudo-sphere $S^{p,q}$, where $p+q=n$, is identified with
the projectivization of the null cone in $\mathbb{R}^{p+1,q+1}$.
The group of conformal transformations of $S^{p,q}$ equals
the projective indefinite orthogonal group $G=\PO(p+1,q+1)=O(p+1,q+1)/\set{\pm 1}$.
Let us use the inner product given by the matrix
\begin{equation}
	\begin{pmatrix}
		& & 1 \\
		& I_{p,q} & \\
		1 & &
	\end{pmatrix},
	\qquad\text{where}\quad
	I_{p,q}=
	\begin{pmatrix}
		I_p & \\
			& I_q
	\end{pmatrix}.
\end{equation}
Then $\transpose{(1,0,\dots,0,0)}$ is a lightlike vector.
Let $P$ denote the isotropy subgroup of $G$ for the line through $\transpose{(1,0,\dots,0,0)}$,
so that $S^{p,q}\cong G/P$ and thus $G$ is a principal $P$-bundle over $S^{p,q}$.
Note that $G$ carries the Maurer--Cartan form $\omega$, a differential 1-form with values in
the Lie algebra $\mathfrak{g}=\mathfrak{o}(p+1,q+1)$.

In general, for any given $n$-dimensional conformal manifold $(M,[g])$ of possibly indefinite signature $(p,q)$,
there is a canonically associated pair of a principal $P$-bundle $\mathcal{G}\to M$
and a $\mathfrak{g}$-valued 1-form $\omega$ on $\mathcal{G}$,
which is known as the \emph{normal Cartan geometry associated with $(M,[g])$}
(see, e.g., \v{C}ap--Slov\'ak \cite{Cap-Slovak-09}*{\S1.6}).

In this context, there is an important grading of the Lie algebra $\mathfrak{g}$,
which is $\mathfrak{g}=\mathfrak{g}_{-1}\oplus\mathfrak{g}_0\oplus\mathfrak{g}_1$,
where $\mathfrak{g}_{-1}\cong\mathbb{R}^n$, $\mathfrak{g}_0\cong\mathfrak{o}(p,q)\oplus\mathbb{R}$,
and $\mathfrak{g}_1\cong(\mathbb{R}^n)^*$.
To be specific, we define each summand as the linear subspace of $\mathfrak{g}$
consisting of elements of the form
\begin{equation}
	\begin{pmatrix}
		& & \\
		X & & \\
		& -\transpose{X} & \phantom{M}
	\end{pmatrix},\qquad
	\begin{pmatrix}
		a & & \\
		& A & \\
		& & -a
	\end{pmatrix},\qquad
	\begin{pmatrix}
		\phantom{M} & Z & \\
		& & -\transpose{Z} \\
		& &
	\end{pmatrix},
\end{equation}
where $X\in\mathbb{R}^n$ (the set of column vectors),
$A\in\mathfrak{o}(p,q)$, $a\in\mathbb{R}$, and $Z\in(\mathbb{R}^n)^*$.
Then, the Lie algebra of the parabolic subgroup $P$ equals $\mathfrak{p}=\mathfrak{g}_0\oplus\mathfrak{g}_1$.

By the properties of Cartan connections,
$\omega$ defines a linear isomorphism $T_u\mathcal{G}\cong\mathfrak{g}$ for all $u\in\mathcal{G}$
by which vertical vectors are mapped onto $\mathfrak{p}$,
and hence each element $X$ of $\mathfrak{g}_{-1}$ defines a vector field $\omega^{-1}(X)$
on the principal bundle $\mathcal{G}$ that is transverse to the fibers.
A curve in $(M,[g])$ is called a \emph{conformal geodesic} if
it is a projection onto $M$ of an integral curve of $\omega^{-1}(X)$ for some $X\in\mathfrak{g}_{-1}$.
In fact, the same definition applies mostly verbatim to any other geometry that fall in the class of
parabolic geometries (see \cite{Cap-Slovak-09}*{\S5.3}).

Herzlich \cite{Herzlich-13} gave an alternative definition of parabolic geodesics
that makes use of so-called Weyl structures,
which specializes to conformal geometry as follows.
Recall that a \emph{Weyl connection} on a conformal manifold $(M,[g])$ is a connection $D$ of $TM$
that can be expressed as $D_XY=\nabla_XY+\alpha(X)Y+\alpha(Y)X-g(X,Y)\alpha^\sharp$,
where $\nabla$ is the Levi-Civita connection of some representative metric in $[g]$ and $\alpha$ is a 1-form.
Then, according to \cite{Herzlich-13},
a conformal geodesic can be reinterpreted as a curve $\gamma\colon I\to M$
satisfying the following properties locally (i.e., in some neighborhood of any $t_0\in I$):
\begin{enumerate}[(i)]
	\item
		\label{item:Herzlich-condition-1}
		There is a Weyl connection on $(M,[g])$ with respect to which $\gamma$ is a geodesic.
	\item
		\label{item:Herzlich-condition-2}
		The same Weyl connection is itself parallel along the curve $\gamma$ in a certain sense
		(that is formulated in terms of the normal Cartan connection).
\end{enumerate}
In fact, the conditions \eqref{eq:conformal-geodesics-in-terms-of-one-form-1}
and \eqref{eq:conformal-geodesics-in-terms-of-one-form-2} that we gave earlier
correspond to \ref{item:Herzlich-condition-1} and \ref{item:Herzlich-condition-2}, respectively.
For details, see the discussion in \cite{Herzlich-13}*{\S6.1}.

It is also possible to eliminate the auxiliary 1-form $\alpha$ from the system
\eqref{eq:system-for-conformal-geodesics-in-terms-of-one-form} in most cases.
Note that \eqref{eq:conformal-geodesics-in-terms-of-one-form-1} implies that
$\nabla_{\dot{\gamma}}(\abs{\dot{\gamma}}^2)+2\alpha(\dot{\gamma})\abs{\dot{\gamma}}^2=0$,
from which it follows that the causal character of $\dot{\gamma}$ is preserved along the curve.
When $\gamma$ is spacelike or timelike,
one can solve \eqref{eq:conformal-geodesics-in-terms-of-one-form-1} for $\alpha$ to obtain
\begin{equation}
	\label{eq:one-form-alpha}
	\alpha^\sharp
	=\frac{1}{\abs{\dot{\gamma}}^2}
	\left(\ddot{\gamma}-\frac{2\braket{\dot{\gamma},\ddot{\gamma}}}{\abs{\dot{\gamma}}^2}\dot{\gamma}\right)
\end{equation}
and deduce that
$\gamma$ is a conformal geodesic if and only if
\begin{equation}
	\label{eq:conformal-geodesics-third-order}
	\dddot{\gamma}-\frac{3\braket{\dot{\gamma},\ddot{\gamma}}}{\abs{\dot{\gamma}}^2}\ddot{\gamma}
	+\frac{3\abs{\ddot{\gamma}}^2}{2\abs{\dot{\gamma}}^2}\dot{\gamma}
	+2P(\dot{\gamma},\dot{\gamma})\dot{\gamma}-\abs{\dot{\gamma}}^2P(\dot{\gamma},\cdot)^\sharp=0.
\end{equation}
If $\gamma$ is lightlike, or null, then such a reduction of
\eqref{eq:system-for-conformal-geodesics-in-terms-of-one-form} is not possible.

We shall observe that conformal geodesics admit projective reparametrizations.
For non-null conformal geodesics, this is discussed in, e.g., Bailey--Eastwood \cite{Bailey-Eastwood-90}.

\begin{prop}
	\label{prop:projective-reparametrizations}
	The set of reparametrizations admitted by any fixed non-constant conformal geodesic
	is exactly the group $\PGL(2,\mathbb{R})$ of linear fractional transformations.
\end{prop}

\begin{proof}
	Suppose that $(\gamma(t),\alpha(t))$ is a solution of
	\eqref{eq:system-for-conformal-geodesics-in-terms-of-one-form},
	and let $\tilde{\gamma}(s)=\gamma(f(s))$.
	If $\tilde{\gamma}(s)$ satisfies \eqref{eq:conformal-geodesics-in-terms-of-one-form-1}
	for some $\tilde{\alpha}(s)$, it implies that
	\begin{equation}
		(f')^2\ddot{\gamma}+f''\dot{\gamma}+2(f')^2\tilde{\alpha}(\dot{\gamma})\dot{\gamma}
		-(f')^2\abs{\dot{\gamma}}^2\tilde{\alpha}^\sharp=0
	\end{equation}
	(where the dots and the primes denote the $t$- and $s$-differentiations, respectively).
	Let $\tilde{\alpha}(s)=\alpha(f(s))+\beta(s)$.
	Then the system \eqref{eq:system-for-conformal-geodesics-in-terms-of-one-form}
	for the pair $(\tilde{\gamma}(s),\tilde{\alpha}(s))$ reads
	\begin{subequations}
		\label{eq:system-for-reparametrizations}
	\begin{gather}
		\label{eq:equation-for-reparametrizations-1}
		f''\dot{\gamma}+2(f')^2\beta(\dot{\gamma})\dot{\gamma}-(f')^2\abs{\dot{\gamma}}^2\beta^\sharp=0, \\
		\label{eq:equation-for-reparametrizations-2}
		\dot{\beta}^\sharp
		-\alpha(\dot{\gamma})\beta^\sharp-\beta(\dot{\gamma})\alpha^\sharp-\beta(\dot{\gamma})\beta^\sharp
		+\left(\braket{\alpha,\beta}+\frac{1}{2}\abs{\beta}^2\right)\dot{\gamma}=0.
	\end{gather}
	\end{subequations}

	When $\abs{\dot{\gamma}}^2\not=0$, \eqref{eq:equation-for-reparametrizations-1} is equivalent to
	$\beta^\sharp=-(f')^{-2}f''\abs{\dot{\gamma}}^{-2}\dot{\gamma}$,
	and putting it into \eqref{eq:equation-for-reparametrizations-2} deduces
	\begin{equation}
		\label{eq:Schwarzian-derivative-vanishing}
		\frac{1}{f'}\left(\frac{f''}{(f')^2}\right)'+\frac{1}{2}\left(\frac{f''}{(f')^2}\right)^2=0.
	\end{equation}
	This means that $f$ must be a linear fractional transformation because
	the left-hand side of \eqref{eq:Schwarzian-derivative-vanishing} is $(f')^{-2}$ times the Schwarzian derivative.

	When $\abs{\dot{\gamma}}^2=0$, \eqref{eq:equation-for-reparametrizations-1} implies that
	$\beta(\dot{\gamma})=-\frac{1}{2}(f')^{-2}f''$, while it follows from
	\eqref{eq:equation-for-reparametrizations-2} that
	$\frac{d}{dt}(\beta(\dot{\gamma}))
	-\beta(\ddot{\gamma})-2\alpha(\dot{\gamma})\beta(\dot{\gamma})-\beta(\dot{\gamma})^2=0$.
	Since $\ddot{\gamma}+2\alpha(\dot{\gamma})\dot{\gamma}=0$ by \eqref{eq:conformal-geodesics-in-terms-of-one-form-1},
	the latter implies
	\begin{equation}
		\label{eq:equation-for-beta-dotgamma-for-null-geodesics}
		\frac{d}{dt}(\beta(\dot{\gamma}))-\beta(\dot{\gamma})^2=0
	\end{equation}
	and by putting $\beta(\dot{\gamma})=-\frac{1}{2}(f')^{-2}f''$ into
	\eqref{eq:equation-for-beta-dotgamma-for-null-geodesics} we obtain \eqref{eq:Schwarzian-derivative-vanishing} again.
	The converse in this case can be checked as follows:
	if $f$ is a linear fractional transformation, then it follows from
	\eqref{eq:equation-for-beta-dotgamma-for-null-geodesics} that
	any local solution $\beta$ of \eqref{eq:equation-for-reparametrizations-2} satisfies
	$\beta(\dot{\gamma})=-\frac{1}{2}(f')^{-2}f''$ as long as this equality holds at some $t=t_0$.
	Hence \eqref{eq:equation-for-reparametrizations-1} is satisfied for such a locally defined 1-form $\beta$,
	and thus $\tilde{\gamma}(s)$ is a conformal geodesic.
\end{proof}

This fact allows us to give the following definition.

\begin{dfn}
	Let $c$ be a curve equipped with a projective structure,
	and $(M,[g])$ a conformal manifold.
	Then a map $\gamma\colon c\to M$ is a \emph{conformal geodesic} if,
	for any projective parametrization $p\colon I\to c$,
	$\gamma\circ p$ is a conformal geodesic.
\end{dfn}

In particular, consider the hyperbolic plane $\mathbb{H}^2$ and
the anti-de Sitter plane $\AdSplane$, which can be both regarded as quotients of hyperboloids
in the 3-dimensional Minkowski space given as
\begin{equation}
	\mathbb{H}^2=\tilde{\mathbb{H}}^2/\mathord{\sim},\qquad
	\tilde{\mathbb{H}}^2=\set{(\xi_0,\xi_1,\xi_2)\in\mathbb{R}^{2,1}|-\xi_0^2+\xi_1^2+\xi_2^2=-1}
\end{equation}
and
\begin{equation}
	\AdSplane=\widetilde{\mathrm{AdS}}_2/\mathord{\sim},\qquad
	\widetilde{\mathrm{AdS}}_2=\set{(\xi_0,\xi_1,\xi_2)\in\mathbb{R}^{1,2}|\xi_0^2-\xi_1^2-\xi_2^2=-1},
\end{equation}
where $\sim$ identifies antipodal points.
In both cases, the \emph{boundary circle at infinity} is the projectivization of the light cone,
which admits a natural action of $\PO(2,1)\cong\PGL(2,\mathbb{R})$.
Then, the both boundary circles carry natural projective structures,
which can be made apparent by representing $\mathbb{H}^2$ and
$\AdSplane\setminus\set{\xi_0+\xi_2=0}$ in the upper-half plane model by letting
\begin{equation}
	s=\frac{1}{\xi_0+\xi_2},\qquad
	t=\frac{\xi_1}{\xi_0+\xi_2}
\end{equation}
on $\tilde{\mathbb{H}}^2\cap\set{\xi_0+\xi_2>0}$
and on $\widetilde{\mathrm{AdS}}_2\cap\set{\xi_0+\xi_2>0}$,
which projects down diffeomorphically to $\mathbb{H}^2$ and to $\AdSplane\setminus\set{\xi_0+\xi_2=0}$, respectively
(see, e.g., Nomizu \cite{Nomizu-82} or Seppi--Trebeschi \cite{Seppi-Trebeschi-22}).
The metrics of $\mathbb{H}^2$ and $\AdSplane$ are expressed as
\begin{equation}
	\label{eq:upper-half-plane-metric-of-hyperbolic-plane}
	h_+=\frac{ds^2+dt^2}{s^2}
\end{equation}
and
\begin{equation}
	\label{eq:upper-half-plane-metric-of-AdS-plane}
	h_+=\frac{ds^2-dt^2}{s^2}
\end{equation}
in this model, and in both cases,
$t$ gives the projective parametrization of $\set{s=0}$, an open subset of the boundary circle.

\section{Harmonic maps and conformal geodesics}
\label{sec:proof-of-main-theorem}

The asymptotic Dirichlet problem for harmonic maps between hyperbolic spaces,
and more generally between conformally compact manifolds, has been
studied by Li--Tam \cite{Li-Tam-91,Li-Tam-93-I,Li-Tam-93-II},
Leung \cite{Leung_91_Thesis},
Economakis \cite{Economakis-93-Thesis}, and
Akutagawa--Matsumoto \cite{Akutagawa-Matsumoto-16}, among others.
The treatment in this section features the explicit asymptotic expansion of harmonic maps to higher orders
than discussed in these works.
While we focus on maps from the hyperbolic plane $\mathbb{H}^2$ into
conformally compact Einstein manifolds $(X,g_+)$,
the method presented here naturally extends to more general settings,
as partly illustrated in Appendix \ref{sec:renormalized-energy}.

\subsection{Preliminaries on conformally compact Einstein manifolds}
\label{sec:prelim-on-ccE-manifolds}

Let $(X,g_+)$ be a conformally compact Einstein manifold of dimension $n+1$ (where $n\geqq 2$)
in the following sense.
We suppose that a compact smooth manifold-with-boundary $\overline{X}$ is given,
and that $g_+$ is a complete smooth Einstein metric on the interior $X$.
We write $\bdry_\infty X$ or $M$ for $\bdry\overline{X}$, which we call the \emph{boundary at infinity}
of $(X,g_+)$.
For any smooth boundary defining function $\rho$,
it is assumed that $\rho^2g_+$ continuously extends to a Riemannian metric of $\overline{X}$,
and that the extension pulls back to a smooth metric on $M$.
Thus $M$ is equipped with an induced conformal class $[(\rho^2g_+)|_{T\bdry\overline{X}}]$ of Riemannian metrics,
hence also the name \emph{conformal infinity}.

In addition, in view of the result of Chru\'sciel--Delay--Lee--Skinner
\cite{Chrusciel-Delay-Lee-Skinner-05},
we assume throughout this paper that, via the identification by some diffeomorphism between
a collar neighborhood of the boundary in $\overline{X}$ and $M\times[0,x_0)$,
the metric $g_+$ can be put in the form
\begin{equation}
	\label{eq:Graham-Lee-normal-form}
	g_+=\frac{dx^2+g_x}{x^2},
\end{equation}
where $g_x$ is a family of Riemannian metrics on $M$ admitting a \emph{polyhomogeneous expansion}
\begin{equation}
	\label{eq:polyhomogeneous-expansion-of-metric}
	g_x\sim g_0+
	\sum_{k=1}^\infty\sum_{l=0}^{N_k}x^{\nu_k}(\log x)^l g_{\nu_k,l}
\end{equation}
where $0<\nu_1<\nu_2<\dotsb\to +\infty$, $N_k$ is a non-negative integer for all $k$,
$g_0$ is a Riemannian metric on $M$, and each $g_{\nu_k,l}$ is a symmetric 2-tensor on $M$.
We mean by \eqref{eq:polyhomogeneous-expansion-of-metric} that for any $k_0$ the truncated series has the property
\begin{equation}
	g_x-\left(g_0+
	\sum_{k=1}^{k_0}\sum_{l=0}^{N_k}x^{\nu_k}(\log x)^l g_{\nu_k,l}\right)
	=o(x^{\nu_{k_0}})
\end{equation}
and a similar asymptotic estimate holds for
the derivatives of $g_x$ of any order, both in the $x$-direction and the directions tangent to $M$.
In the expansion \eqref{eq:polyhomogeneous-expansion-of-metric}, the terms $x^{\nu_k}(\log x)^l g_{\nu_k,l}$ will
be called \emph{of order $\nu_k$}.

The metric $g_0$, which we also write $g$ in the sequel,
represents the induced conformal class on $M$.
The Fermi-type normalization \eqref{eq:Graham-Lee-normal-form} of $g_+$ will be termed
a \emph{Graham--Lee normalization} in what follows because of the work \cite{Graham-Lee-91},
and the function $x$ in \eqref{eq:Graham-Lee-normal-form} is called
the \emph{special boundary defining function} of $(X,g_+)$
associated with this normalization.

In fact, it is straightforward from the Einstein equation
written down for the expression \eqref{eq:Graham-Lee-normal-form}
(see, e.g., Graham \cite{Graham-00} or Fefferman--Graham \cite{Fefferman-Graham-12})
that $g_x$ takes the form
\begin{equation}
	\label{eq:Einstein-metric-expansion}
	g_x=g+xg_1+x^2g_2+\dotsb+x^{n-1}g_{n-1}+x^ng_n+x^n\log x\cdot h+o(x^n),
\end{equation}
and a more careful analysis shows that $g_x$ is forced to take the form
\begin{equation}
	\label{eq:Einstein-metric-even-expansion}
	g_x=
	\begin{cases}
		g+x^2g_2+o(x^2) & (\text{$n=2$}), \\
		g+x^2g_2+\dotsb+(\text{even powers})+\dotsb+x^{n-1}g_{n-1}+x^ng_n+o(x^n) & (\text{$n$ odd}), \\
		g+x^2g_2+\dotsb+(\text{even powers})+\dotsb+x^ng_n+x^n\log x\cdot h+o(x^n) & (\text{$n$ even $\geqq 4$}),
	\end{cases}
\end{equation}
which we call an \emph{even} expansion.
A derivation of \eqref{eq:Einstein-metric-even-expansion} can be outlined as follows.
The process of getting \eqref{eq:Einstein-metric-expansion} shows that
each truncated series $g^{(n-1)}_x=g+xg_1+\dots+x^{n-1}g_{n-1}$ is the unique polynomial in $x$ of
degree $n-1$ satisfying $\Ric(g_+)=-ng_++O(x^n)$.
By formally replacing the variable $x$ with $-x$,
one obtains $g_{-x}^{(n-1)}=g-xg_1+\dots+(-1)^{n-1}x^{n-1}g_{n-1}$,
which necessarily solves the same approximate Einstein equation because of the diffeomorphism invariance
of the Ricci tensor\footnote{The expression
$g_+=x^{-2}(dx^2+g^{(n-1)}_x)$ gives a metric satisfying $\Ric(g_+)=-ng_++O(x^n)$
defined not only for $x>0$ but also for $x<0$.
Consequently, the invariance of $\Ric$ with respect to the diffeomorphism $x\mapsto -x$ implies that
$g_-=x^{-2}(dx^2+g^{(n-1)}_{-x})$ defined for $x>0$ satisfies $\Ric(g_-)=-ng_-+O(x^n)$ as well.}.
The uniqueness then implies that $g_1=g_3=\dots=0$.
The vanishing of $h$ follows again from parity reasons for odd $n$,
and for $n=2$ this is ultimately because $\Ric(g)$ is pure trace in this dimension.

Regarding the higher-order terms that are not displayed in \eqref{eq:Einstein-metric-even-expansion},
it can be checked that $g_x$ does not contain any terms of non-integer order.
Another important remark is that, when $n$ is odd, the expansion does not contain any logarithmic terms.
In other words, $g_x$ is actually smooth (i.e., $C^\infty$) in $x$ up to $x=0$.

The coefficients $g_k$ for $k<n$ and $\tr_gg_n$ are \emph{formally determined}
in the sense that they can be written down explicitly in terms of $g$, its curvature tensor, and
the covariant derivatives thereof.
For example, it is known that $g_2=-P$ when $n\geqq 3$, where $P$ is the Schouten tensor of $g$,
while $\tr_gg_2=-\frac{1}{2}R$ for $n=2$.
We also remark that $\tr_gg_n$ vanishes for $n$ odd.
For $n\geqq 4$ even, the log-term coefficient $h$ is formally determined as well and trace-free,
which is called the \emph{Fefferman--Graham obstruction tensor} of $(M,[g])$.
The trace-free part of $g_n$ (for both $n$ odd and even) is \emph{formally undetermined}.

The knowledge in the literature regarding
the existence of a genuine conformally compact Einstein filling $(X,g_+)$ for a given $(M,[g])$ is limited.
However, a classical result of Fefferman--Graham \cite{Fefferman-Graham-85,Fefferman-Graham-12} states that,
if we regard \eqref{eq:polyhomogeneous-expansion-of-metric} as a formal expansion,
then the existence of such an expansion with $\Ric(g_+)=-ng_++o(x^n)$ for any given $(M,[g])$ is
always guaranteed.
We adopt this viewpoint for the metric $g_+$ in the sequel.
As an added benefit, we can also allow $[g]$ to be a conformal class of pseudo-Riemannian metrics,
since the result of Fefferman--Graham remains true for such conformal classes on $M$.

\subsection{Preliminaries on harmonic maps and outline of the proof}
\label{subsec:prelim-on-harmonic-maps}

Next, let $(Y^{q+1},h_+)$ and $(X^{n+1},g_+)$ be two conformally compact Einstein manifolds,
whose conformal infinities are denoted by $N$ and $M$, respectively.
(We will immediately specialize to the case where $Y=\mathbb{H}^2$ or $Y=\AdSplane$.)
We put the metrics into the Graham--Lee normal form \eqref{eq:Graham-Lee-normal-form} as
\begin{equation}
	\label{eq:Graham-Lee-normal-form-for-two-metrics}
	h_+=\frac{ds^2+h_s}{s^2},\qquad
	g_+=\frac{dx^2+g_x}{x^2},
\end{equation}
where $h_s\sim h+s^2h_2+\dotsb$ and $g_x\sim g+x^2g_2+\dotsb$ are polyhomogeneous.
We take the case $q=1$ into our consideration,
in which case we always assume that $h_+$ equals either
\eqref{eq:upper-half-plane-metric-of-hyperbolic-plane} or \eqref{eq:upper-half-plane-metric-of-AdS-plane},
while we continue to assume that $n\geqq 2$ for the target manifold.
Furthermore, $h_s$ and $g_x$ can also be regarded as formal polyhomogeneous expansions.

Consider a proper smooth mapping $u\colon Y\to X$ admitting a polyhomogeneous expansion,
whose precise meaning is as follows.
We take local coordinates $(t^a)$ in $N$ and $(y^i)$ in $M$,
and write $u$, locally, as a set of functions
\begin{equation}
	(u^0,u^i)=(u^0(s,t^a),u^i(s,t^a)).
\end{equation}
Our assumption is that the functions $u^0$ and $u^i$ have expansions
\begin{equation}
	\label{eq:polyhomogeneous-expansion-of-map}
	u^0\sim
	\sum_{k=1}^\infty\sum_{l=0}^{N_k}s^{\nu_k}(\log s)^l u^0_{\nu_k,l},\qquad
	u^i\sim u_0^i+
	\sum_{k=1}^\infty\sum_{l=0}^{N_k}s^{\nu_k}(\log s)^l u^i_{\nu_k,l}.
\end{equation}
Each $u^0_{\nu_k,l}$ represents a globally-defined smooth function on $N$,
while $u_0^i$ represents a smooth mapping $N\to M$, which is nothing but the restriction of $u$ to the boundary.
We also write $\varphi=\set{\varphi^i}_{i=1}^n$ for the latter mapping.
Again, we can also take \eqref{eq:polyhomogeneous-expansion-of-map} as formal expansions,
which is indispensable when $h_s$ and $g_x$ are considered formally.

The mapping $u$ is harmonic when the tension field $\tau(u)=\tr_{h_+}\nabla du$ vanishes.
In local expression, $\tau(u)$ is given by
\begin{equation}
	\label{eq:formula-of-tension-field}
	\begin{split}
	\tau(u)^I
	&=\Delta_{h_+}u^I+h_+^{AB}(\tensor{\Gamma}{^I_J_K}\circ u)(\partial_Au^J)(\partial_Bu^K)\\
	&=h_+^{AB}\partial_A\partial_Bu^I
	-h_+^{AB}\tensor{\Gamma}{^C_A_B}\partial_Cu^I
	+h_+^{AB}(\tensor{\Gamma}{^I_J_K}\circ u)(\partial_Au^J)(\partial_Bu^K),
	\end{split}
\end{equation}
where $A$, $B$, $\dotsc\in\set{0,1,\dots,q}$,
$I$, $J$, $\dotsc\in\set{0,1,\dots,n}$,
and $\tensor{\Gamma}{^C_A_B}$ and $\tensor{\Gamma}{^I_J_K}$ are the Christoffel symbols of
$h_+$ and $g_+$, respectively.
If we assume that $u$ is polyhomogeneous and $\tau(u)=o(1)$, or equivalently, $\tau(u)^I=o(s)$,
and moreover that $\varphi$ has nowhere vanishing differential,
then it can be checked that $u^0$ and $u^i$ carry no terms of order $0<\nu<1$ and
\begin{equation}
	\label{eq:first-normal-coefficient-of-harmonic-map}
	u_1^0=\frac{\abs{d\varphi}}{\sqrt{q}},\qquad
	u_1^i=0,
\end{equation}
where $\abs{d\varphi}$ is the pointwise norm with respect to $h$ and $g$
(see \S\ref{subsec:first-order-coefficients} for the case $Y=\mathbb{H}^2$,
and compare also with Li--Tam \cite{Li-Tam-93-I}*{Lemma 1.3} and
Akutagawa--Matsumoto \cite{Akutagawa-Matsumoto-16}*{Lemma 4}).
A further straightforward computation
shows that the components of $\tau(u)$ with respect to $(x,y^i)$ take the form
\begin{subequations}
	\label{eq:harmonic-map-equation-principal-part}
\begin{align}
	\label{eq:harmonic-map-equation-principal-part-normal}
	\tau(u)^0&=(s\partial_s)^2u^0-(q+2)(s\partial_s)u^0-(q-1)u^0+R(u)^0,\\
	\label{eq:harmonic-map-equation-principal-part-tangential}
	\tau(u)^i&=(s\partial_s)^2u^i-(q+2)(s\partial_s)u^i+R(u)^i,
\end{align}
\end{subequations}
with remainder terms $R(u)^0$ and $R(u)^i$ satisfying
\begin{equation}
	\tilde{u}^I=u^I+O(s^\nu(\log s)^l)\quad\Longrightarrow\quad
	R(\tilde{u})^I-R(u)^I=o(s^\nu).
\end{equation}
These formulae allow us to investigate polyhomogeneous expansion of a harmonic map $u$ in more detail.

Now we switch to the special cases where $Y=\mathbb{H}^2$ or $Y=\AdSplane$,
in which the metric $h_+$ is given by \eqref{eq:upper-half-plane-metric-of-hyperbolic-plane}
or \eqref{eq:upper-half-plane-metric-of-AdS-plane}.
Recall also that, if $n\geqq 3$, the Einstein condition for
\begin{equation}
	\label{eq:Graham-Lee-form-for-target}
	g_+=\frac{dx^2+g_x}{x^2}
\end{equation}
implies
\begin{equation}
	\label{eq:second-order-expansion-of-target-metric}
	g_x=g-x^2P+o(x^2),
\end{equation}
where $P$ is the Schouten tensor of $g$.
We write $g_2=-P$ when $n=2$ for brevity. In both cases, $P$ is a symmetric 2-tensor satisfying
\begin{equation}
	\tr P=-\frac{1}{2}R,\qquad
	\nabla^j P_{ij}=\frac{1}{2}\nabla_iR,
\end{equation}
$R$ being the scalar curvature of $g$ (see Fefferman--Graham \cite{Fefferman-Graham-12}*{Theorem 3.7}).

In this setting, the formulae \eqref{eq:harmonic-map-equation-principal-part} for the tension field specialize to
\begin{subequations}
	\label{eq:harmonic-map-equation-principal-part-for-surfaces}
\begin{align}
	\label{eq:harmonic-map-equation-principal-part-for-surfaces-normal}
	\tau(u)^0&=(s\partial_s)^2u^0-3(s\partial_s)u^0+R(u)^0,\\
	\label{eq:harmonic-map-equation-principal-part-for-surfaces-tangential}
	\tau(u)^i&=(s\partial_s)^2u^i-3(s\partial_s)u^i+R(u)^i.
\end{align}
\end{subequations}
Furthermore, since the expansion of $g_x$ consists only of integer-order terms,
further computation shows that, if we set $\tilde{u}^I=u^I+s^\nu(\log s)^lv^I$ for a positive integer $\nu$,
\begin{equation}
	\label{eq:change-of-tension-field-for-surfaces}
	\begin{split}
		&\tau(\tilde{u})^I-\tau(u)^I \\
		&=
		\begin{cases}
			\nu(\nu-3)s^\nu v^I+S^I & \text{if $l=0$}, \\
			\nu(\nu-3)s^\nu\log s\cdot v^I+(2\nu-3)s^\nu v^I+S^I & \text{if $l=1$}, \\
			\nu(\nu-3)s^\nu(\log s)^lv^I+(2\nu-3)ls^\nu(\log s)^{l-1}v^I
			+l(l-1)s^\nu(\log s)^{l-2}v^I+S^I & \text{if $l\geqq 2$},
		\end{cases}
	\end{split}
\end{equation}
where $S^I$ consists of terms of integer order $\geqq\nu+1$.
In view of this, the first part of our main theorem
(and its analog for the case $Y=\AdSplane$) can be shown immediately as follows.

\begin{proof}[Proof of Theorem \ref{thm:main} (1)]
	The given boundary curve $\gamma\colon\bdry_\infty Y\to M$ can be understood
	as the set $(\gamma^1(t),\gamma^2(t),\dots,\gamma^n(t))$ of functions using local coordinates $(y^i)$ in $M$.
	We define $u_{(0)}$, the zeroth-order approximate solution for the harmonic map equation, by
	$u_{(0)}^0(s,t)=0$ and $u_{(0)}^i(s,t)=\gamma^i(t)$.
	Starting with this, for $K=1$, $2$, $3$, $\dotsc$, we inductively construct a formal map $u_{(K)}$ of the form
	\begin{equation}
		u_{(K)}^0(s,t)=
		\sum_{k=1}^K\sum_{l=0}^{N_k}s^k(\log s)^l u^0_{k,l}(t),\qquad
		u_{(K)}^i(s,t)=
		\gamma^i(t)+
		\sum_{k=1}^K\sum_{l=0}^{N_k}s^k(\log s)^l u^i_{k,l}(t)
	\end{equation}
	satisfying $\tau(u_{(K)})^I=o(s^K)$.
	Assuming we have such $u_{(K-1)}$, if we set
	\begin{equation}
		u_{(K)}^0(s,t)=u_{(K-1)}^0(s,t)
		+\sum_ls^K(\log s)^l v^0_l(t),\qquad
		u_{(K)}^i(s,t)=u_{(K-1)}^i(s,t)
		+\sum_ls^K(\log s)^l v^i_l(t),
	\end{equation}
	formula \eqref{eq:change-of-tension-field-for-surfaces} implies that
	the terms $s^K(\log s)^l v^I_l$ that should be introduced to get $\tau(u_{(K)})^I=o(s^K)$ are
	uniquely determined provided $K\not=3$.
	When $K=3$, while $v^0_0$ ($=u^0_{3,0}$) and $v^i_0$ ($=u^i_{3,0}$) are undetermined,
	all the other logarithmic terms can be uniquely determined so that $\tau(u_{(3)})^I=o(s^3)$ is satisfied.
	(As we shall see below,
	$K=3$ is in fact the first step in which we may need to introduce logarithmic terms.)
\end{proof}

In what follows, we determine the expansion more explicitly.
Recall the general formula \eqref{eq:formula-of-tension-field} of the tension field.
Since the Laplacian on functions is conformally covariant in dimension $q+1=2$, in the current setting we have
\begin{equation}
	\label{eq:tension-field-for-maps-from-plane}
	\tau(u)^I
	=s^2\bigg[\partial_s^2u^I\pm\partial_t^2u^I
	+\tensor{\Gamma}{^I_J_K}(u(s,t))(\partial_su^J)(\partial_su^K)
	\pm\tensor{\Gamma}{^I_J_K}(u(s,t))(\partial_tu^J)(\partial_tu^K)\bigg],
\end{equation}
where we take the upper (resp.\ the lower) sign of the two plus-minus signs
for $u\colon\mathbb{H}^2\to(X,g_+)$ (resp.\ for $u\colon\AdSplane\to(X,g_+)$).
We will write $\tau(u)_{\mathbb{H}^2}^I$ and $\tau(u)_{\AdSplane}^I$ when a distinction
between the two tension fields must be made notationally.
Using \eqref{eq:second-order-expansion-of-target-metric}, we can derive that
\begin{equation}
	\label{eq:Christoffel-symbols-for-g+}
	\begin{gathered}
		\tensor{\Gamma}{^0_0_0}=-\frac{1}{x},\qquad
		\tensor{\Gamma}{^0_j_k}=\frac{1}{x}\tensor{g}{_j_k}+o(x),\qquad
		\tensor{\Gamma}{^0_0_k}=0,\\
		\tensor{\Gamma}{^i_0_k}=-\frac{1}{x}\tensor{\delta}{_k^i}-x\tensor{P}{_k^i}+o(x),\qquad
		\tensor{\Gamma}{^i_0_0}=0,\\
		\tensor{\Gamma}{^i_j_k}=\tensor{(\Gamma^g)}{^i_j_k}
		-\frac{1}{2}x^2(\nabla_j\tensor{P}{_k^i}+\nabla_k\tensor{P}{_j^i}-\nabla^i\tensor{P}{_j_k})+o(x^2),
	\end{gathered}
\end{equation}
where $\tensor{(\Gamma^g)}{^i_j_k}$ is the Christoffel symbol for $(M,g)$.

Detailed computations in the rest of this section
based on \eqref{eq:tension-field-for-maps-from-plane} and \eqref{eq:Christoffel-symbols-for-g+},
we will see that, if we require $\tau(u)=o(s^2)$, or equivalently $\tau(u)^I=o(s^3)$,
then $u$ must take the form
\begin{equation}
	\label{eq:polyhomogeneous-expansion-in-action}
	\begin{alignedat}{13}
		u^0(s,t) &= \hspace{1.5pt}x(s,t) &&= && \hspace{12pt}sx_1(t) && +s^2x_2(t) && +s^3x_3(t) && +s^3\log s\cdot v^0(t) && +o(s^3), \\
		u^i(s,t)\hspace{.5pt} &= y^i(s,t) &&= \gamma^i(t) && +sy_1^i(t) && +s^2y_2^i(t) && +s^3y_3^i(t) && +s^3\log s\cdot v^i(t) && +o(s^3)
	\end{alignedat}
\end{equation}
and the coefficients $x_1$, $y_1^i$, $x_2$, $y_2^i$, $v^0$, and $v^i$ are formally determined
(in a strong sense, i.e., in the sense that they are given by a local expression in terms of $\gamma$ and $g$).
The remaining coefficients, $x_3$ and $y_3^i$, are undetermined.

We will see in \S\ref{subsec:second-fundamental-form} that
if we additionally impose $\nabla du=o(s^2)$, $x_3$ and $y_3^i$ are fixed
and we obtain one more equality that must be satisfied by the coefficients,
hence by the curve $\gamma$, which happens to be the conformal geodesic equation
\eqref{eq:conformal-geodesics-third-order}.
More precisely, the following theorem is established in \S\ref{subsec:second-fundamental-form},
thereby completing the proof of Theorem \ref{thm:main} (2).

\begin{thm}
	\label{thm:main-theorem-in-action}
	Let $g_+$ be a conformally compact Einstein metric \eqref{eq:Graham-Lee-form-for-target},
	where $g_x$ is given formally, with $g=g_0$ possibly of indefinite signature.
	Suppose that $u$ is a formal polyhomogeneous map
	of the form \eqref{eq:polyhomogeneous-expansion-in-action}
	from $\mathbb{H}^2$ such that $\gamma$ is everywhere spacelike,
	or $u$ is such a map
	from $\AdSplane$ such that $\gamma$ is everywhere timelike.
	Then, if we set
	\begin{equation}
		\abs{\dot{\gamma}}^2=\abs{\dot{\gamma}}^2_g=\braket{\dot{\gamma},\dot{\gamma}}_g
		=\pm\lambda^2,\qquad \lambda=\lambda(t)>0,
	\end{equation}
	$\tau(u)^I=o(s^3)$ holds if and only if
	\begin{equation}
		\label{eq:formally-determined-coefficients-unified-form}
		x_1=\lambda,\qquad
		y_1^i=0,\qquad
		x_2=0,\qquad
		y_2^i=\frac{1}{2}\lambda^2\alpha^i,\qquad
		v^0=0,\qquad
		v^i=0,
	\end{equation}
	where $\alpha^i$ is defined by \eqref{eq:one-form-alpha}.
	Moreover,
	$\nabla du=o(s^2)$ holds if and only if $x_3=-\frac{1}{4}\lambda^3\abs{\alpha}^2$, $y_3^i=0$,
	and $\gamma$ is a conformal geodesic.
\end{thm}

For later computations, we remark that
\begin{equation}
	\label{eq:lambda-derivative}
	\partial_t\lambda=\pm\frac{\braket{\dot{\gamma},\ddot{\gamma}}}{\lambda}=-\lambda\alpha(\dot{\gamma})
\end{equation}
and hence
\begin{equation}
	\label{eq:lambda-twice-derivative}
	\begin{split}
		\partial_t^2\lambda
		&=-(\partial_t\lambda)\alpha(\dot{\gamma})
		-\lambda(\nabla_{\dot{\gamma}}\alpha)(\dot{\gamma})-\lambda\alpha(\ddot{\gamma})
		=\lambda\alpha(\dot{\gamma})^2
		-\lambda(\nabla_{\dot{\gamma}}\alpha)(\dot{\gamma})-\lambda\alpha(\ddot{\gamma})\\
		&=3\lambda\alpha(\dot{\gamma})^2
		-\lambda(\nabla_{\dot{\gamma}}\alpha)(\dot{\gamma})\mp\lambda^3\abs{\alpha}^2.
	\end{split}
\end{equation}

\subsection{First-order coefficients}
\label{subsec:first-order-coefficients}

The first-order coefficients in \eqref{eq:polyhomogeneous-expansion-in-action} are
already known as in \eqref{eq:first-normal-coefficient-of-harmonic-map}
for the case when $g$ is positive definite.
However, we recover this result based on the computation specialized in our current setting,
and at the same time we also generalize it to the case of indefinite conformal infinity.

First, we compute $\tau(u)^0$ modulo $o(s)$ as follows, starting from
\eqref{eq:tension-field-for-maps-from-plane}:
\begin{equation}
	\begin{split}
		\tau(u)^0
		&\equiv s^2\bigg[\tensor{\Gamma}{^0_0_0}\cdot((\partial_sx)^2\pm(\partial_tx)^2)
		+\tensor{\Gamma}{^0_j_k}\cdot(\partial_sy^j\cdot\partial_sy^k\pm\partial_ty^j\cdot\partial_ty^k)\bigg]\\
		&\equiv s^2\left[\left(-\frac{1}{x}\right)((\partial_sx)^2\pm(\partial_tx)^2)
			+\frac{1}{x}g_{jk}(\partial_sy^j\cdot\partial_sy^k\pm\partial_ty^j\cdot\partial_ty^k)\right]\\
		&\equiv s^2\cdot\frac{1}{sx_1+o(s)}(-x_1^2+\abs{y_1}^2\pm\abs{\dot{\gamma}}^2+o(1))\\
		&\equiv s\cdot\frac{1}{x_1+o(1)}(-x_1^2+\abs{y_1}^2\pm\abs{\dot{\gamma}}^2+o(1)).
	\end{split}
\end{equation}
If this expression is $o(s)$, then it is necessary that
\begin{equation}
	\lim_{s\to 0}\frac{1}{x_1+o(1)}(-x_1^2+\abs{y_1}^2\pm\abs{\dot{\gamma}}^2+o(1))=0,
\end{equation}
from which we can deduce that $x_1$ must be nowhere vanishing.
Using this, the above computation can be continued and we obtain
\begin{equation}
	\label{eq:tau-0-modulo-terms-higher-than-first-order}
	\tau(u)^0
	\equiv s\cdot\frac{1}{x_1}(-x_1^2+\abs{y_1}^2\pm\abs{\dot{\gamma}}^2).
\end{equation}

For $\tau(u)^i$, we can compute (using that $x_1$ is nowhere vanishing) as follows, modulo $o(s)$:
\begin{equation}
	\label{eq:tau-i-modulo-terms-higher-than-first-order}
	\begin{split}
		\tau(u)^i
		&\equiv s^2\bigg[2\tensor{\Gamma}{^i_0_k}\cdot(\partial_sx\cdot\partial_sy^k\pm\partial_tx\cdot\partial_ty^k)
		+\tensor{\Gamma}{^i_j_k}\cdot(\partial_sy^j\cdot\partial_sy^k\pm\partial_ty^j\cdot\partial_ty^k))\bigg]\\
		&\equiv s^2\bigg[2\cdot \left(-\frac{1}{x}\tensor{\delta}{_k^i}\right)(x_1y_1^k\pm 0)\bigg]
		\equiv s^2\cdot 2\cdot \left(-\frac{1}{sx_1}\tensor{\delta}{_k^i}\right)x_1y_1^k
		\equiv -2sy_1^i.
	\end{split}
\end{equation}

It follows from \eqref{eq:tau-0-modulo-terms-higher-than-first-order} and
\eqref{eq:tau-i-modulo-terms-higher-than-first-order}
that it is necessary for $\tau(u)^I=o(s)$ to hold that $\gamma$ is a spacelike curve when $Y=\mathbb{H}^2$,
and that $\gamma$ is a timelike curve when $Y=\AdSplane$.
Assuming this necessary condition is satisfied, $\tau(u)^I=o(s)$ if and only if
the first-order coefficients are given by
\begin{equation}
	\label{eq:first-order-coefficients}
	x_1=\lambda,\qquad y_1^i=0,
\end{equation}
where we write $\abs{\dot{\gamma}}^2=\pm\lambda^2$, $\lambda>0$, as already defined.

\begin{rem}
	Given the result for the case $Y=\mathbb{H}^2$, we could also argue as follows to draw the conclusion for
	the case $Y=\AdSplane$.
	Note that we get the anti-de Sitter metric \eqref{eq:upper-half-plane-metric-of-AdS-plane} from
	the hyperbolic metric \eqref{eq:upper-half-plane-metric-of-hyperbolic-plane}
	by formally replacing $s$ with $\pm\sqrt{-1}\,s$.
	The same is true for the formula \eqref{eq:tension-field-for-maps-from-plane} of the tension field.
	This leads to the following fact: if
	\begin{equation}
		\label{eq:first-order-solution-to-harmonic-map-equation-for-spacelike-case}
		\begin{alignedat}{3}
			x(s,t) &= && \hspace{12pt}sx_1(t)+o(s), \\
			y^i(s,t) &= \gamma^i(t) && +sy_1^i(t)+o(s)
		\end{alignedat}
	\end{equation}
	is a solution to the equation $\tau(u)_{\mathbb{H}^2}^I=o(s)$, then
	\begin{equation}
		\label{eq:first-order-solution-to-harmonic-map-equation-for-spacelike-case-with-wick-rotated}
		\begin{alignedat}{5}
			x(s,t) &= && \pm\sqrt{-1}\,sx_1(t) && +o(s), \\
			y^i(s,t) &= \gamma^i(t) && \pm\sqrt{-1}\,sy_1^i(t) && +o(s)
		\end{alignedat}
	\end{equation}
	formally solves the equation $\tau(u)_{\AdSplane}^I=o(s)$
	(where we are free to choose the $+$ or the $-$ sign, provided we make the same choice for the two lines).
	Now observe that, even when $\gamma$ is timelike, if we set
	\begin{equation}
		x_1=\sqrt{-1}\,\lambda,\qquad y_1^i=0
	\end{equation}
	so that $x_1^2=\abs{\dot{\gamma}}^2$, then this is formally a solution to
	the equation $\tau(u)_{\mathbb{H}^2}^I=o(s)$.
	By taking the $-$ sign in
	\eqref{eq:first-order-solution-to-harmonic-map-equation-for-spacelike-case-with-wick-rotated},
	we obtain the solution
	\begin{equation}
		\begin{alignedat}{11}
			x(s,t) &= && -\sqrt{-1}\cdot s\sqrt{-1}\,\lambda && +o(s) &&= && \hspace{12pt}s\lambda && +o(s), \\
			y^i(s,t) &= \gamma^i(t) && -\sqrt{-1}\,s\cdot 0 && +o(s) &&= \gamma^i(t) && && +o(s)
		\end{alignedat}
	\end{equation}
	to $\tau(u)_{\AdSplane}^I=o(s)$. The fact that this is the only solution can be shown in a similar vein.
\end{rem}

\subsection{Second-order coefficients}

Next, we shall determine $x_2$ and $y_2^i$ so that $\tau(u)^I=o(s^2)$.
Modulo $o(s^2)$ we compute
\begin{equation}
	\begin{split}
		\tau(u)^0
		&\equiv s^2\bigg[\partial_s^2x\pm\partial_t^2x
		+\tensor{\Gamma}{^0_0_0}\cdot((\partial_sx)^2\pm(\partial_tx)^2)
		+\tensor{\Gamma}{^0_j_k}\cdot(\partial_sy^j\cdot\partial_sy^k\pm\partial_ty^j\cdot\partial_ty^k)\bigg]\\
		&\equiv s^2\left[2x_2\pm 0+\left(-\frac{1}{x}\right)((\partial_sx)^2\pm(\partial_tx)^2)
		+\frac{1}{x}g_{jk}(\partial_sy^j\cdot\partial_sy^k\pm\partial_ty^j\cdot\partial_ty^k)\right]\\
		&\equiv s^2\left[2x_2-\frac{1}{s\lambda+s^2x_2}((\lambda+2sx_2)^2\pm 0)
		+\frac{1}{s\lambda+s^2x_2}g_{jk}(0\pm\dot{\gamma}^j\dot{\gamma}^k)\right]\\
		&\equiv s^2\left[2x_2-\frac{1}{s\lambda+s^2x_2}((\lambda+2sx_2)^2-\lambda^2)\right]\\
		&\equiv s^2\left[2x_2-\frac{1}{s\lambda+s^2x_2}\cdot 4s\lambda x_2\right]
		\equiv -2s^2x_2.
	\end{split}
\end{equation}
On the other hand, $\tau(u)^i$ modulo $o(s^2)$ is given by
\begin{equation}
	\tau(u)^i
	\equiv s^2\bigg[\partial_s^2y^i\pm\partial_t^2y^i
	+2\tensor{\Gamma}{^i_0_k}\cdot(\partial_sx\cdot\partial_sy^k\pm\partial_tx\cdot\partial_ty^k)
	+\tensor{\Gamma}{^i_j_k}\cdot(\partial_sy^j\cdot\partial_sy^k\pm\partial_ty^j\cdot\partial_ty^k)\bigg].
\end{equation}
By formula \eqref{eq:Christoffel-symbols-for-g+} and the fact that $x=O(s)$,
the third term in the right-hand side bracket equals
$2(-x^{-1}\tensor{\delta}{_k^i})(\partial_sx\cdot\partial_sy^k\pm\partial_tx\cdot\partial_ty^k)$ modulo $o(1)$,
while the fourth term equals
$\tensor{(\Gamma^g)}{^i_j_k}(y(s,t))\cdot(\partial_sy^j\cdot\partial_sy^k\pm\partial_ty^j\cdot\partial_ty^k)$
modulo $o(1)$.
Moreover, we have $\tensor{(\Gamma^g)}{^i_j_k}(y(s,t))=\tensor{(\Gamma^g)}{^i_j_k}(\gamma(t))$ modulo $o(1)$.
Therefore we get, by writing $\tensor{(\Gamma^g)}{^i_j_k}$ for $\tensor{(\Gamma^g)}{^i_j_k}(\gamma(t))$,
\begin{equation}
	\begin{split}
		\tau(u)^i
		&\equiv s^2\bigg[2y_2^i\pm\partial_t\dot{\gamma}^i
		-\frac{2}{s\lambda}\tensor{\delta}{_k^i}(\lambda\cdot 2sy_2^k\pm(s\partial_t\lambda)\cdot\dot{\gamma}^k)
		+\tensor{(\Gamma^g)}{^i_j_k}(0\pm\dot{\gamma}^j\dot{\gamma}^k)\bigg]\\
		&\equiv s^2\bigg[2y_2^i\pm\partial_t\dot{\gamma}^i
		-4y_2^i\mp 2\frac{\partial_t\lambda}{\lambda}\cdot\dot{\gamma}^i
		\pm\tensor{(\Gamma^g)}{^i_j_k}\dot{\gamma}^j\dot{\gamma}^k\bigg]
		\equiv s^2\bigg[{-2y_2^i}\pm\ddot{\gamma}^i
		\pm 2\alpha(\dot{\gamma})\dot{\gamma}^i\bigg]
	\end{split}
\end{equation}
by \eqref{eq:lambda-derivative}.
Consequently, $\tau(u)^I=o(s^2)$ holds if and only if $x_2=0$ and
\begin{equation}
	y_2^i
	=\pm\frac{1}{2}(\ddot{\gamma}^i+2\alpha(\dot{\gamma})\dot{\gamma}^i)
	=\pm\frac{1}{2}\abs{\dot{\gamma}}^2\alpha^i
	=\frac{1}{2}\lambda^2\alpha^i.
\end{equation}

\subsection{Third-order coefficients}
\label{subsec:third-order-coefficients}

Equation \eqref{eq:harmonic-map-equation-principal-part-for-surfaces} shows that
there is no contribution of $x_3$ and $y_3^i$ to the $s^3$ term in $\tau(u)^I$, which we recover below.
In fact, we will also see that $\tau(u)^I=o(s^3)$ is already satisfied by the coefficients determined so far,
and that the log-term coefficients $v^0$ and $v^i$ must be zero.

Recall that
\begin{equation}
	\tau(u)^0
	=s^2\bigg[\partial_s^2x\pm\partial_t^2x
	+\tensor{\Gamma}{^0_0_0}\cdot((\partial_sx)^2\pm(\partial_tx)^2)
	+\tensor{\Gamma}{^0_j_k}\cdot(\partial_sy^j\cdot\partial_sy^k\pm\partial_ty^j\cdot\partial_ty^k)\bigg].
\end{equation}
Note that $\tensor{\Gamma}{^0_j_k}$ in this formula actually means $\tensor{\Gamma}{^0_j_k}(u(s,t))$,
which must be computed as
\begin{equation}
	\begin{split}
		\tensor{\Gamma}{^0_j_k}(u(s,t))
		&=\frac{1}{x}g_{jk}(y(s,t))+o(s)
		=\frac{1}{x}g_{jk}(\gamma(t)+s^2y_2(t))+o(s) \\
		&=\frac{1}{x}\left(g_{jk}(\gamma(t))+s^2\frac{\partial g_{jk}}{\partial y^l}y^l_2(t)\right)+o(s).
	\end{split}
\end{equation}
(We substitute $\frac{1}{2}\lambda^2\alpha^l$ for $y_2^l$ later.)
Therefore we have, modulo $o(s^3)$,
\begin{multline}
	\label{eq:tau-0-modulo-third-order}
	\tau(u)^0
	\equiv s^2\bigg[6sx_3\pm s\partial_t^2\lambda-\frac{1}{s\lambda}((\partial_sx)^2\pm(\partial_tx)^2)\\
	\hspace{3em}
	+\frac{1}{s\lambda}\left(g_{jk}+s^2\frac{\partial g_{jk}}{\partial y^l}y_2^l\right)
	(\partial_sy^j\cdot\partial_sy^k\pm\partial_ty^j\cdot\partial_ty^k)\bigg].
\end{multline}
Consequently,
\begin{equation}
	\begin{split}
		\tau(u)^0
		&\equiv
		s^2\bigg[6sx_3\pm s\partial_t^2\lambda
		-\frac{1}{s\lambda}((\lambda+3s^2x_3)^2\pm(s\partial_t\lambda)^2)\\
		&\hspace{2em}
		+\frac{1}{s\lambda}\left(g_{jk}+s^2\frac{\partial g_{jk}}{\partial y^l}y_2^l\right)
		(s\lambda^2\alpha^j\cdot s\lambda^2\alpha^k
		\pm(\dot{\gamma}^j+s^2\partial_ty_2^j)(\dot{\gamma}^k+s^2\partial_ty_2^k))\bigg]\\
		&\equiv
		s^2\bigg[6sx_3\pm s\partial_t^2\lambda
		-\frac{1}{s\lambda}(\lambda^2+6s^2\lambda x_3\pm s^2(\partial_t\lambda)^2)\\
		&\hspace{2em}
		+\frac{1}{s\lambda}\left(g_{jk}+s^2\frac{\partial g_{jk}}{\partial y^l}y_2^l\right)
		(s^2\lambda^4\alpha^j\alpha^k
		\pm(\dot{\gamma}^j\dot{\gamma}^k+s^2\dot{\gamma}^j\partial_ty_2^k+s^2\dot{\gamma}^k\partial_ty_2^j))\bigg]\\
		&\equiv
		s^2\bigg[{\pm s\partial_t^2\lambda}-\frac{1}{s\lambda}(\lambda^2\pm s^2(\partial_t\lambda)^2)\\
		&\hspace{2em}
		+\frac{1}{s\lambda}(s^2\lambda^4\abs{\alpha}^2+\lambda^2\pm 2s^2\braket{\dot{\gamma},\partial_ty_2})
		\pm\frac{1}{s\lambda}\cdot s^2\frac{\partial g_{jk}}{\partial y^l}y_2^l\cdot
		\dot{\gamma}^j\dot{\gamma}^k\bigg].
	\end{split}
\end{equation}
Using the fact that
\begin{equation}
	\label{eq:Christoffel-symbol-formula}
	\frac{1}{2}\frac{\partial g_{jk}}{\partial y^l}\cdot\dot{\gamma}^j\dot{\gamma}^k
	=g_{jm}\tensor{\Gamma}{^m_k_l}\dot{\gamma}^j\dot{\gamma}^k,
\end{equation}
we obtain
\begin{equation}
	\begin{split}
		\tau(u)^0
		&\equiv
		s^3\bigg[{\pm\partial_t^2\lambda}\mp\frac{1}{\lambda}(\partial_t\lambda)^2
		+\frac{1}{\lambda}(\lambda^4\abs{\alpha}^2\pm 2\braket{\dot{\gamma},\nabla_{\dot{\gamma}}y_2})\bigg]\\
		&\equiv
		s^3\bigg[{\pm\partial_t^2\lambda}\mp\frac{1}{\lambda}(\partial_t\lambda)^2
		+\lambda^3\abs{\alpha}^2
		\pm 2(\partial_t\lambda)\alpha(\dot{\gamma})
		\pm\lambda(\nabla_{\dot{\gamma}}\alpha)(\dot{\gamma})\bigg],
	\end{split}
\end{equation}
which vanishes by \eqref{eq:lambda-derivative} and \eqref{eq:lambda-twice-derivative}.
On the other hand, $\tau(u)^i$ modulo $o(s^3)$ is given by
\begin{equation}
	\label{eq:tau-i-modulo-third-order}
	\begin{split}
		\tau(u)^i
		&\equiv
		s^2\bigg[\partial_s^2y^i\pm\partial_t^2y^i
		+2\tensor{\Gamma}{^i_0_k}\cdot(\partial_sx\cdot\partial_sy^k\pm\partial_tx\cdot\partial_ty^k)
		+\tensor{\Gamma}{^i_j_k}\cdot(\partial_sy^j\cdot\partial_sy^k\pm\partial_ty^j\cdot\partial_ty^k)\bigg]\\
		&\equiv
		s^2\bigg[(2y_2^i+6sy_3^i)\pm\partial_t\dot{\gamma}^i
		+2\left(-\frac{1}{x}\tensor{\delta}{_k^i}-x\tensor{P}{_k^i}\right)
		(\partial_sx\cdot\partial_sy^k\pm\partial_tx\cdot\partial_ty^k)\\
		&\hspace{2em}
		+\tensor{(\Gamma^g)}{^i_j_k}(y(s,t))\cdot
		(\partial_sy^j\cdot\partial_sy^k\pm\partial_ty^j\cdot\partial_ty^k)\bigg].
	\end{split}
\end{equation}
Since $\partial_sx\cdot\partial_sy^k+\partial_tx\cdot\partial_ty^k$ is $O(s)$,
the contribution of $-x\tensor{P}{_k^i}$ is negligible.
Therefore, by noting also that $\partial_sy^j\cdot\partial_sy^k=O(s^2)$,
\begin{equation}
	\begin{split}
		\tau(u)^i
		&\equiv
		s^2\bigg[(2y_2^i+6sy_3^i)\pm\partial_t\dot{\gamma}^i
		-\frac{2}{s\lambda}\tensor{\delta}{_k^i}
		(\lambda\cdot(2sy_2^k+3s^2y_3^k)\pm(s\partial_t\lambda)\cdot\dot{\gamma}^k)
		\pm\tensor{(\Gamma^g)}{^i_j_k}\dot{\gamma}^j\dot{\gamma}^k\bigg]\\
		&\equiv
		s^2\bigg[2y_2^i\pm\partial_t\dot{\gamma}^i
		-4y_2^i\mp 2\frac{\partial_t\lambda}{\lambda}\dot{\gamma}^i
		\pm\tensor{(\Gamma^g)}{^i_j_k}\dot{\gamma}^j\dot{\gamma}^k\bigg]
		\equiv
		s^2\bigg[{\pm\ddot{\gamma}^i}-\lambda^2\alpha^i\mp 2\frac{\partial_t\lambda}{\lambda}\dot{\gamma}^i\bigg]
	\end{split}
\end{equation}
and this vanishes again.
Consequently, $\tau(u)^I=o(s^3)$ holds no matter how one set $x_3$ and $y_3^i$.

Moreover, if we let
\begin{equation}
	\begin{alignedat}{13}
		u^0 &= \hspace{.5pt}x &&= && \hspace{12pt}s\lambda && && +s^3x_3 && +s^3\log s\cdot v^0 && +o(s^3), \\
		u^i\hspace{.5pt} &= y^i &&= \gamma^i && && +\frac{1}{2}s^2\lambda^2\alpha^i && +s^3y_3^i && +s^3\log s\cdot v^i && +o(s^3),
	\end{alignedat}
\end{equation}
then \eqref{eq:tau-0-modulo-third-order} becomes
\begin{multline}
	\tau(u)^0
	\equiv s^2\bigg[(6sx_3+5sv^0+6s\log s\cdot v^0)\pm s\partial_t^2\lambda
	-\frac{1}{s\lambda}((\partial_sx)^2\pm(\partial_tx)^2)\\
	+\frac{1}{s\lambda}\left(g_{jk}+s^2\frac{\partial g_{jk}}{\partial y^l}y_2^l\right)
	(\partial_sy^j\cdot\partial_sy^k\pm\partial_ty^j\cdot\partial_ty^k)\bigg].
\end{multline}
By continuing the computation as before, we obtain $\tau(u)^0\equiv 3s^3v^0$,
and hence $v^0$ must be zero in order for $\tau(u)^0=o(s^3)$.
Similarly, \eqref{eq:tau-i-modulo-third-order} becomes
\begin{multline}
	\tau(u)^i
	\equiv s^2\bigg[(2y_2^i+6sy_3^i+5sv^i+6s\log s\cdot v^i)\pm\partial_t\dot{\gamma}^i\\
	+2\left(-\frac{1}{x}\tensor{\delta}{_k^i}-x\tensor{P}{_k^i}\right)
	(\partial_sx\cdot\partial_sy^k\pm\partial_tx\cdot\partial_ty^k)\\
	+\tensor{(\Gamma^g)}{^i_j_k}(y(s,t))\cdot(\partial_sy^j\cdot\partial_sy^k\pm\partial_ty^j\cdot\partial_ty^k)\bigg]
\end{multline}
and we obtain $\tau(u)^i\equiv 3s^3v^i$,
which implies that $v^i$ should vanish in order for $\tau(u)^i=o(s^3)$.

\subsection{Second fundamental form}
\label{subsec:second-fundamental-form}

We have seen that $u$ satisfies $\tau(u)^I=o(s^3)$ if and only if
\eqref{eq:formally-determined-coefficients-unified-form} is satisfied.
To complete the proof of Theorem \ref{thm:main-theorem-in-action},
we compute the second fundamental form $\nabla du$ for such mappings $u$ modulo $o(s^2)$,
or equivalently, the components $\tensor{(\nabla du)}{_A_B^I}$ modulo $o(s)$,
where $A$, $B\in\set{s,t}$ are the indices associated with the upper-half plane coordinates on $Y$.

If we write $\Phi$ for the differential $du$ of the map
\begin{equation}
	\label{eq:third-order-critical-mapping}
	\begin{alignedat}{11}
		u^0 &= \hspace{.5pt}x &&= && \hspace{12pt}s\lambda && && +s^3x_3 && +o(s^3), \\
		u^i\hspace{.5pt} &= y^i &&= \gamma^i && && +\frac{1}{2}s^2\lambda^2\alpha^i && +s^3y_3^i && +o(s^3),
	\end{alignedat}
\end{equation}
modulo $o(s^2)$ we have
\begin{alignat}{2}
	\tensor{\Phi}{_s^0}&\equiv \lambda+3s^2x_3,&\qquad
	\tensor{\Phi}{_t^0}&\equiv s\partial_t\lambda\equiv -s\lambda\alpha(\dot{\gamma}),\\
	\tensor{\Phi}{_s^i}&\equiv s\lambda^2\alpha^i+3s^2y_3^i,&\qquad
	\tensor{\Phi}{_t^i}&\equiv \dot{\gamma}^i+\frac{1}{2}s^2\partial_t(\lambda^2\alpha^i)
	\equiv\dot{\gamma}^i-s^2\lambda^2\alpha(\dot{\gamma})\alpha^i
	+\frac{1}{2}s^2\lambda^2\partial_t\alpha^i.
\end{alignat}
Based on these formulae, we compute each component of $\nabla du=\nabla\Phi$ modulo $o(s)$ using
\begin{equation}
	\nabla_A\tensor{\Phi}{_B^I}
	=\partial_A\tensor{\Phi}{_B^I}-\tensor{\Gamma}{^C_A_B}\tensor{\Phi}{_C^I}
	+\tensor{\Gamma}{^I_J_K}\tensor{\Phi}{_A^J}\tensor{\Phi}{_B^K}.
\end{equation}
Note that $\tensor{\Gamma}{^C_A_B}$, the Christoffel symbol for $Y$, is given explicitly by
\begin{alignat}{3}
	\tensor{\Gamma}{^s_s_s}&=-\frac{1}{s},&\qquad
	\tensor{\Gamma}{^s_s_t}&=\tensor{\Gamma}{^s_t_s}=0,&\qquad
	\tensor{\Gamma}{^s_t_t}&=\pm\frac{1}{s},\\
	\tensor{\Gamma}{^t_s_s}&=0,&\qquad
	\tensor{\Gamma}{^t_s_t}&=\tensor{\Gamma}{^t_t_s}=-\frac{1}{s},&\qquad
	\tensor{\Gamma}{^t_t_t}&=0,
\end{alignat}
where the upper sign is for $Y=\mathbb{H}^2$ and the lower one is for $Y=\AdSplane$.

\begin{lem}
	\label{lem:second-fundamental-form}
	The components $\nabla_A\tensor{\Phi}{_B^I}$ are given by, modulo $o(s)$,
	\begin{gather}
		\nabla_s\tensor{\Phi}{_s^0}
		\equiv s(4x_3+\lambda^3\abs{\alpha}^2),\qquad
		\nabla_t\tensor{\Phi}{_t^0}
		\equiv \mp s(4x_3+\lambda^3\abs{\alpha}^2),\\
		\nabla_s\tensor{\Phi}{_t^0}
		\equiv \nabla_t\tensor{\Phi}{_s^0}
		\equiv 3s\lambda^{-1}\braket{\dot{\gamma},y_3},\\
		\nabla_s\tensor{\Phi}{_s^i}
		\equiv 3sy_3^i,\qquad
		\nabla_t\tensor{\Phi}{_t^i}
		\equiv \mp 3sy_3^i,\\
		\nabla_s\tensor{\Phi}{_t^i}
		\equiv\nabla_t\tensor{\Phi}{_s^i}
		\equiv s(-2\lambda^{-1}x_3\dot{\gamma}^i
		+\lambda^2(\nabla_{\dot{\gamma}}\alpha)^i
		-\lambda^2\alpha(\dot{\gamma})\alpha^i
		-\lambda^2\tensor{P}{_k^i}\dot{\gamma}^k).
	\end{gather}
\end{lem}

\begin{proof}
	We omit the computation for
	$\nabla_s\tensor{\Phi}{_t^0}$, $\nabla_s\tensor{\Phi}{_s^i}$ and $\nabla_t\tensor{\Phi}{_t^i}$,
	which are relatively easy.

	Everything is calculated modulo $o(s)$ in what follows.
	Then, first, the component $\nabla_s\tensor{\Phi}{_s^0}
	=\partial_s\tensor{\Phi}{_s^0}
	-\tensor{\Gamma}{^s_s_s}\tensor{\Phi}{_s^0}
	+\tensor{\Gamma}{^0_0_0}\tensor{\Phi}{_s^0}\tensor{\Phi}{_s^0}
	+\tensor{\Gamma}{^0_j_k}\tensor{\Phi}{_s^j}\tensor{\Phi}{_s^k}$ equals
	\begin{equation}
		\partial_s\tensor{\Phi}{_s^0}-\left(-\frac{1}{s}\right)\tensor{\Phi}{_s^0}
		+\left(-\frac{1}{x}\right)\tensor{\Phi}{_s^0}\tensor{\Phi}{_s^0}
		+\frac{1}{x}g_{jk}(y(s,t))\cdot\tensor{\Phi}{_s^j}\tensor{\Phi}{_s^k}
	\end{equation}
	or
	\begin{equation}
		6sx_3-\left(-\frac{1}{s}\right)(\lambda+3s^2x_3)
		+\left(-\frac{1}{s\lambda+s^3x_3}\right)(\lambda+3s^2x_3)^2
		+\frac{1}{s\lambda}g_{jk}\cdot s^2\lambda^4\alpha^j\alpha^k,
	\end{equation}
	which simplifies to $s(4x_3+\lambda^3\abs{\alpha}^2)$.
	
	Next,
	$\nabla_t\tensor{\Phi}{_t^0}
	=\partial_t\tensor{\Phi}{_t^0}
	-\tensor{\Gamma}{^s_t_t}\tensor{\Phi}{_s^0}
	+\tensor{\Gamma}{^0_0_0}\tensor{\Phi}{_t^0}\tensor{\Phi}{_t^0}
	+\tensor{\Gamma}{^0_j_k}\tensor{\Phi}{_t^j}\tensor{\Phi}{_t^k}$ equals
	\begin{equation}
		\partial_t\tensor{\Phi}{_t^0}-\left(\pm\frac{1}{s}\right)\cdot\tensor{\Phi}{_s^0}
		+\left(-\frac{1}{x}\right)\tensor{\Phi}{_t^0}\tensor{\Phi}{_t^0}
		+\frac{1}{x}g_{jk}(y(s,t))\cdot\tensor{\Phi}{_t^j}\tensor{\Phi}{_t^k},
	\end{equation}
	which is equal to
	\begin{multline}
		{-s\partial_t(\lambda\alpha(\dot{\gamma}))}\mp\frac{1}{s}(\lambda+3s^2x_3)
		+\left(-\frac{1}{s\lambda}\right)\cdot s^2\lambda^2\alpha(\dot{\gamma})^2 \\
		+\frac{1}{s\lambda+s^3x_3}
		\left(g_{jk}+\frac{\partial g_{jk}}{\partial y^l}\cdot\frac{s^2}{2}\lambda^2\alpha^l\right)
		(\dot{\gamma}^j\dot{\gamma}^k-2s^2\lambda^2\alpha(\dot{\gamma})\dot{\gamma}^j\alpha^k
		+s^2\lambda^2\dot{\gamma}^j\partial_t\alpha^k).
	\end{multline}
	Consequently, using \eqref{eq:Christoffel-symbol-formula} we obtain
	\begin{multline}
		\nabla_t\tensor{\Phi}{_t^0}
		\equiv
		-s(\partial_t\lambda)\alpha(\dot{\gamma})
		-s\lambda(\nabla_{\dot{\gamma}}\alpha)(\dot{\gamma})
		-s\lambda\alpha(\ddot{\gamma})
		\mp\frac{1}{s}(\lambda+3s^2x_3)
		+\left(-\frac{1}{s\lambda}\right)\cdot s^2\lambda^2\alpha(\dot{\gamma})^2 \\
		+\frac{1}{s\lambda+s^3x_3}
		\cdot g_{jk}
		(\dot{\gamma}^j\dot{\gamma}^k-2s^2\lambda^2\alpha(\dot{\gamma})\dot{\gamma}^j\alpha^k
		+s^2\lambda^2\dot{\gamma}^j\partial_t\alpha^k)
		+\frac{1}{s\lambda}\cdot s^2\lambda^2\alpha^l\cdot
		g_{jm}\tensor{\Gamma}{^m_k_l}\dot{\gamma}^j\dot{\gamma}^k.
	\end{multline}
	We can simplify this using \eqref{eq:lambda-derivative} to get
	$\nabla_t\tensor{\Phi}{_t^0}\equiv \mp s(4x_3+\lambda^3\abs{\alpha}^2)$.
	
	Lastly,
	$\nabla_s\tensor{\Phi}{_t^i}
	=\partial_s\tensor{\Phi}{_t^i}
	-\tensor{\Gamma}{^t_s_t}\tensor{\Phi}{_t^i}
	+\tensor{\Gamma}{^i_0_k}\tensor{\Phi}{_s^0}\tensor{\Phi}{_t^k}
	+\tensor{\Gamma}{^i_0_k}\tensor{\Phi}{_t^0}\tensor{\Phi}{_s^k}
	+\tensor{\Gamma}{^i_j_k}\tensor{\Phi}{_s^j}\tensor{\Phi}{_t^k}$ equals
	\begin{equation}
		\partial_s\tensor{\Phi}{_t^i}
		-\left(-\frac{1}{s}\right)\tensor{\Phi}{_t^i}
		+\left(-\frac{1}{x}\tensor{\delta}{_k^i}-x\tensor{P}{_k^i}\right)\tensor{\Phi}{_s^0}\tensor{\Phi}{_t^k}
		+\left(-\frac{1}{x}\tensor{\delta}{_k^i}-x\tensor{P}{_k^i}\right)\tensor{\Phi}{_t^0}\tensor{\Phi}{_s^k}
		+\tensor{(\Gamma^g)}{^i_j_k}\tensor{\Phi}{_s^j}\tensor{\Phi}{_t^k},
	\end{equation}
	which equals
	\begin{multline}
		(-2s\lambda^2\alpha(\dot{\gamma})\alpha^i+s\lambda^2\partial_t\alpha^i)
		+\frac{1}{s}
		\left(\dot{\gamma}^i-s^2\lambda^2\alpha(\dot{\gamma})\alpha^i+\frac{1}{2}s^2\lambda^2\partial_t\alpha^i\right)
		\\
		-\left(\frac{1}{s\lambda+s^3x_3}\tensor{\delta}{_k^i}+s\lambda\tensor{P}{_k^i}\right)
		(\lambda+3s^2x_3)
		\left(\dot{\gamma}^k-s^2\lambda^2\alpha(\dot{\gamma})\alpha^k
		+\frac{1}{2}s^2\lambda^2\partial_t\alpha^k\right) \\
		-\frac{1}{s\lambda}\tensor{\delta}{_k^i}
		\cdot(-s\lambda\alpha(\dot{\gamma}))\cdot s\lambda^2\alpha^k
		+\tensor{(\Gamma^g)}{^i_j_k}\cdot s\lambda^s\alpha^j\cdot\dot{\gamma}^k.
	\end{multline}
	It is straightforward to see that this expression equals the one in the last line of the lemma.
\end{proof}

By Lemma \ref{lem:second-fundamental-form}, $\nabla_A\tensor{\Phi}{_B^I}=O(s)$ is always true,
and $\nabla_A\tensor{\Phi}{_B^I}=o(s)$ if and only if
$4x_3+\lambda^3\abs{\alpha}^2=0$, $y_3^i=0$, and
\begin{equation}
	\label{eq:asymptotic-total-geodesicness-second}
	-2\lambda^{-1}x_3\dot{\gamma}^i
	+\lambda^2(\nabla_{\dot{\gamma}}\alpha)^i
	-\lambda^2\alpha(\dot{\gamma})\alpha^i
	-\lambda^2\tensor{P}{_k^i}\dot{\gamma}^k=0
\end{equation}
are satisfied.
The first equality shows that $x_3=-\frac{1}{4}\lambda^3\abs{\alpha}^2$.
Then the last equality is equivalent to
\begin{equation}
	\frac{1}{2}\abs{\alpha}^2\dot{\gamma}^i
	+(\nabla_{\dot{\gamma}}\alpha)^i
	-\alpha(\dot{\gamma})\alpha^i
	-\tensor{P}{_k^i}\dot{\gamma}^k=0,
\end{equation}
which is nothing but \eqref{eq:conformal-geodesics-in-terms-of-one-form-2}.
In view of the equivalence between \eqref{eq:conformal-geodesics-in-terms-of-one-form-1} and
\eqref{eq:one-form-alpha}, the latter of which is the definition of $\alpha$ in the current context,
this finishes the proof of Theorem \ref{thm:main-theorem-in-action},
and hence that of Theorem \ref{thm:main} (2).

\subsection{Asymptotic isometricity}

To close the circle of ideas, let us compute $u^*g_+$ for the map $u$ that we have constructed.
We will do this by allowing $o(s^2)$ errors with respect to $h_+$,
or equivalently, $o(1)$ errors in each coefficient of the expression in terms of $ds$ and $dt$.
This happens to be nothing but the computation omitted in the proof of \cite{Fine-Herfray-22}*{Theorem 5}.
To reveal some subtleties, we begin by assuming only that $u$ is given by \eqref{eq:third-order-critical-mapping},
with letting $x_3$ and $y_3^i$ free.

From \eqref{eq:second-order-expansion-of-target-metric}, we have
\begin{equation}
	u^*g_+=\frac{dx^2+(g_{ij}(y)-x^2P_{ij}(y)+o(x^2))dy^idy^j}{x^2}.
\end{equation}
For the mapping given by \eqref{eq:third-order-critical-mapping}, this becomes
\begin{multline}
	\label{eq:pullback-metric-for-spacelike-case}
	u^*g_+
	\equiv
	\frac{1}{s^2\lambda^2(1+s^2\lambda^{-1}x_3)^2}
	\bigg[(\lambda\,ds+s\partial_t\lambda\,dt+3s^2x_3\,ds)^2\\
	+\left(g_{ij}+\frac{\partial g_{ij}}{\partial y^k}\cdot \frac{s^2}{2}\lambda^2\alpha^k-s^2\lambda^2P_{ij}\right)dy^idy^j\bigg]
\end{multline}
modulo allowed errors, where
\begin{equation}
	dy^i
	\equiv\dot{\gamma}^idt+s\lambda^2\alpha^ids+3s^2y_3^ids+s^2\lambda(\partial_t\lambda)\alpha^idt
	+\frac{1}{2}s^2\lambda^2\partial_t\alpha^idt.
\end{equation}
We compute that
\begin{multline}
	g_{ij}\,dy^idy^j
	\equiv
	\pm\lambda^2dt^2+s\lambda^2\alpha(\dot{\gamma})(ds\,dt+dt\,ds)\\
	+s^2(2\lambda(\partial_t\lambda)\alpha(\dot{\gamma})
	+\lambda^2g_{ij}\dot{\gamma}^i\partial_t\alpha^j)dt^2
	+s^2\lambda^4\abs{\alpha}^2ds^2
	+3s^2\braket{\dot{\gamma},y_3}(ds\,dt+dt\,ds),
\end{multline}
and
\begin{equation}
	\left(\frac{\partial g_{ij}}{\partial y^k}\cdot \frac{s^2}{2}\lambda^2\alpha^k-s^2\lambda^2P_{ij}\right)dy^idy^j
	\equiv
	s^2\lambda^2
	(g_{jl}\dot{\gamma}^j\tensor{\Gamma}{^l_i_k}\dot{\gamma}^i\alpha^k-P_{ij}\dot{\gamma}^i\dot{\gamma}^j)dt^2.
\end{equation}
Consequently, the second term in the bracket in the right-hand side of \eqref{eq:pullback-metric-for-spacelike-case} is,
modulo allowed errors,
\begin{multline}
	{\pm\lambda^2dt^2}+s\lambda^2\alpha(\dot{\gamma})(ds\,dt+dt\,ds)\\
	+s^2(2\lambda(\partial_t\lambda)\alpha(\dot{\gamma})
	+\lambda^2(\nabla_{\dot{\gamma}}\alpha)(\dot{\gamma}))dt^2
	+s^2\lambda^4\abs{\alpha}^2ds^2
	-s^2\lambda^2P_{ij}\dot{\gamma}^i\dot{\gamma}^jdt^2\\
	+3s^2\braket{\dot{\gamma},y_3}(ds\,dt+dt\,ds).
\end{multline}
Therefore,
\begin{equation}
	\begin{split}
		u^*g_+
		&\equiv
		\frac{1}{s^2\lambda^2}\bigg[
		\lambda^2ds^2+s\lambda\partial_t\lambda(ds\,dt+dt\,ds)+s^2(\partial_t\lambda)^2dt^2+6s^2\lambda x_3\,ds^2\\
		&\hspace{4em}\pm\lambda^2dt^2+s\lambda^2\alpha(\dot{\gamma})(ds\,dt+dt\,ds)\\
		&\hspace{4em}+s^2(2\lambda(\partial_t\lambda)\alpha(\dot{\gamma})
		+\lambda^2(\nabla_{\dot{\gamma}}\alpha)(\dot{\gamma}))dt^2
		+s^2\lambda^4\abs{\alpha}^2ds^2-s^2\lambda^2P_{ij}\dot{\gamma}^i\dot{\gamma}^jdt^2\\
		&\hspace{4em}+3s^2\braket{\dot{\gamma},y_3}(ds\,dt+dt\,ds)\bigg] \\
		&\hspace{2em}-\frac{2x_3}{\lambda}(ds^2\pm dt^2),
	\end{split}
\end{equation}
and by \eqref{eq:lambda-derivative} we obtain
\begin{multline}
	\label{eq:pullback-metric-calculated}
	u^*g_+
	\equiv
	\frac{ds^2\pm dt^2}{s^2}
	+\frac{1}{\lambda^2}\bigg[
	(4\lambda x_3+\lambda^4\abs{\alpha}^2)ds^2\\
	+((\partial_t\lambda)^2
	\mp 2\lambda x_3
	+2\lambda(\partial_t\lambda)\alpha(\dot{\gamma})
	+\lambda^2(\nabla_{\dot{\gamma}}\alpha)(\dot{\gamma})
	-\lambda^2P_{ij}\dot{\gamma}^i\dot{\gamma}^j)dt^2+3\braket{\dot{\gamma},y_3}(ds\,dt+dt\,ds)\bigg].
\end{multline}

It follows that $u^*g_+-h_+$ is $o(s^2)$ with respect to $h_+$ if and only if
$4x_3+\lambda^3\abs{\alpha}^2=0$, $y_3^i=0$, and
\begin{equation}
	\label{eq:asymptotic-isometricity-second}
	(\partial_t\lambda)^2
	\mp 2\lambda x_3
	+2\lambda(\partial_t\lambda)\alpha(\dot{\gamma})
	+\lambda^2(\nabla_{\dot{\gamma}}\alpha)(\dot{\gamma})
	-\lambda^2P_{ij}\dot{\gamma}^i\dot{\gamma}^j=0
\end{equation}
are satisfied.
If we set $x_3=-\frac{1}{4}\lambda^3\abs{\alpha}^2$,
then \eqref{eq:asymptotic-isometricity-second} is equivalent to
\begin{equation}
	\label{eq:asymptotic-isometricity-second-rewritten}
	\pm\frac{1}{2}\lambda^2\abs{\alpha}^2
	-\alpha(\dot{\gamma})^2
	+(\nabla_{\dot{\gamma}}\alpha)(\dot{\gamma})
	-P(\dot{\gamma},\dot{\gamma})=0.
\end{equation}
This is equation \eqref{eq:conformal-geodesics-in-terms-of-one-form-2} restricted in the direction
tangent to $\gamma$. In particular, we have shown the following result.

\begin{thm}
	Let $g_+$ be a conformally compact Einstein metric \eqref{eq:Graham-Lee-form-for-target},
	where $g_x$ is given as a formal polyhomogeneous series, with $g=g_0$ possibly indefinite.
	Suppose that $u$ is a proper mapping from $Y=\mathbb{H}^2$ (resp.\ $Y=\AdSplane$) to $(X,g_+)$,
	also given as formal polyhomogeneous series,
	such that $\gamma\colon\bdry_\infty Y\to M$ is everywhere spacelike (resp.\ timelike).
	If $u$ satisfies $\tau(u)=o(s^2)$,
	then it is always true that $u^*g_+-h_+=O(s^2)$.
	If moreover $\nabla du=o(s^2)$ holds, then $u^*g_+-h_+=o(s^2)$.
\end{thm}

In the approach of Fine--Herfray \cite{Fine-Herfray-22} based on the renormalized area of surfaces,
the tangential component of \eqref{eq:conformal-geodesics-in-terms-of-one-form-2} is obtained this way,
i.e., imposing the isometricity modulo $o(s^2)$ as a condition that the parametrization of
the $\Aren$-critical surface is required to satisfy.
In fact, the second fundamental form condition
in \eqref{eq:isometric-totally-geodesic-to-third-order-Fine-Herfray} of Theorem \ref{thm:Fine-Herfray} only provides
the normal component of \eqref{eq:conformal-geodesics-in-terms-of-one-form-2},
and hence it should be paired with the asymptotic isometricity condition.
This is not the case in our approach.

\appendix

\section{Renormalized energy}
\label{sec:renormalized-energy}

We discuss the notion of the renormalized energy of proper harmonic maps between
conformally compact Einstein manifolds.
It will be observed that the formal polyhomogeneous harmonic map $u\colon\mathbb{H}^2\to X$
satisfying $\nabla du=o(e^{-2r})$ is necessarily formally critical with respect to the renormalized energy $\Eren$.

\subsection{Renormalized volume and area}

Let us first recall the renormalized volume and area defined by
Henningson--Skenderis \cite{Henningson-Skenderis-98},
Graham--Witten \cite{Graham-Witten-99}, and Graham \cite{Graham-00}.

The asymptotic behavior of the volume of the region
$X_\varepsilon=\set{x>\varepsilon}$ as $\varepsilon\to +0$ can be described as follows.
From \eqref{eq:Einstein-metric-even-expansion} and the fact that $g_n$ is trace-free for $n$ odd, it can be seen that
\begin{equation}
	\frac{dV_{g_x}}{dV_g}=
	1+\sum_{\substack{2\leqq k\leqq n \\ \text{$k$ even}}}x^kv_k+o(x^n),
\end{equation}
where all the displayed coefficients $v_k$ are formally determined.
Since $dV_{g_+}=x^{-n-1}dx\,dV_{g_x}$, we obtain the following
asymptotic expansion of the volume of $X_\varepsilon$,
in which the constant term is called the \emph{renormalized volume} of $(X,g_+)$ and denoted by $\Vren$:
\begin{equation}
	\int_{X_\varepsilon}dV_{g_+}
	=
	\begin{cases}
		\displaystyle
		\sum_{\nu=0}^{(n-1)/2}\dfrac{V_{2\nu}}{\varepsilon^{n-2\nu}}+\Vren+o(1) & (\text{$n$ odd}), \\[10pt]
		\displaystyle
		\sum_{\nu=0}^{n/2-1}\dfrac{V_{2\nu}}{\varepsilon^{n-2\nu}}
		+L\log\dfrac{1}{\varepsilon}+\Vren+o(1) & (\text{$n$ even}).
	\end{cases}
\end{equation}
For $n$ odd, it is known that $\Vren$
is invariant under changes of the normalization \eqref{eq:Graham-Lee-normal-form}
because of the following lemma,
while for $n$ even we cannot expect the same invariance due to the existence of the log term
(but alternatively, $L$ becomes an invariant).

\begin{lem}[Graham \cite{Graham-00}]
	\label{lem:Graham-lemma-for-changing-GL-normalization}
	For $n$ odd,
	let $x$ and $\Hat{x}$ be the special boundary defining functions associated with two different
	Graham--Lee normalizations \eqref{eq:Graham-Lee-normal-form}. Then
	\begin{equation}
		\Hat{x}=xb
	\end{equation}
	for a smooth function $b$ on $[0,x_0)\times M$ such that $b|_{\set{0}\times M}$ is positive everywhere
	whose Taylor expansion at $x=0$ consists only of
	even powers of $x$ up to (and including) the $x^{n+1}$ term.
\end{lem}

The renormalized area of a minimal submanifold $Y$ in a conformally compact Einstein manifold $(X,g_+)$
is defined similarly.
First, let $N\subset M$ be an arbitrary submanifold of dimension $q\geqq 1$ in the boundary,
and we take local coordinates
$(t^1,\dots,t^q,u^1,\dots,u^{n-q})$ for $M$ so that $u=0$ defines $N$ and
$\partial_{t^i}\perp\partial_{u^j}$ with respect to the induced conformal class on $M$.
Then we consider submanifolds $\overline{Y}\subset\overline{X}$ of dimension $q+1$ that meet $M$ transversely
and satisfy $\overline{Y}\cap M=N$.
Taking a Graham--Lee normalization,
we may express $\overline{Y}$, near the boundary, as a graph of a function $u(t,x)$
satisfying the boundary condition $u(t,0)=0$.
Now we impose the condition that $Y=\overline{Y}\cap X$ should be a minimal submanifold of $X$.
If we presuppose that $u$ has a polyhomogeneous expansion, then by appealing to the minimal submanifold equation,
one finds that
\begin{equation}
	u=
	\begin{cases}
		x^2u_2+\dotsb+(\text{even powers})+\dotsb+x^{q+1}u_{q+1}+x^{q+2}u_{q+2}+o(x^{q+2}) & (\text{$q$ odd}), \\
		x^2u_2+\dotsb+(\text{even powers})+\dotsb+x^{q+2}u_{q+2}+x^{q+2}\log x\cdot v+o(x^{q+2}) & (\text{$q$ even}).
	\end{cases}
\end{equation}
After some further calculations, one can find that the asymptotic behavior of
the $(q+1)$-dimensional area of $Y_\varepsilon=Y\cap\set{x>\varepsilon}$ as $\varepsilon\to +0$ is
\begin{equation}
	\int_{Y_\varepsilon}dA
	=
	\begin{cases}
		\displaystyle
		\sum_{\nu=0}^{(q-1)/2}\dfrac{A_{2\nu}}{\varepsilon^{n-2\nu}}+\Aren+o(1) & (\text{$q$ odd}), \\[10pt]
		\displaystyle
		\sum_{\nu=0}^{q/2-1}\dfrac{A_{2\nu}}{\varepsilon^{n-2\nu}}+K\log\dfrac{1}{\varepsilon}+\Aren+o(1) & (\text{$q$ even}).
	\end{cases}
\end{equation}
As before, $\Aren$ is called the \emph{renormalized area} of $Y$, which is in fact independent of
the normalization \eqref{eq:Graham-Lee-normal-form} for $q$ odd.

\subsection{Renormalized energy}
\label{subsec:renormalized-energy}

Let $(Y^{q+1},h_+)$ and $(X^{n+1},g_+)$ be two conformally compact Einstein manifolds,
where $q\geqq 1$ and $n\geqq 2$.
Our notation here follows \S\ref{subsec:prelim-on-harmonic-maps},
and we assume that $h_s$ and $g_x$ admit polyhomogeneous expansions.
We do not allow $h_s$ and $g_x$ to be mere formal polyhomogeneous series.

Generally, a proper smooth mapping $u\colon Y\to X$ that extends to
a continuous mapping $\overline{Y}\to\overline{X}$ may have infinite Dirichlet energy
\begin{equation}
	\mathcal{E}(u)=\frac{1}{2}\int_Y\abs{du}^2dV_{h_+}.
\end{equation}
To introduce an energy renormalization, we assume that $u$ admits polyhomogeneous expansion
\eqref{eq:polyhomogeneous-expansion-of-map}.
Furthermore, we need the evenness and the regularity of the expansion in the sense of
the following definition.

\begin{dfn}
	\label{dfn:evenness-of-maps}
	Let $u\colon (Y^{q+1},h_+)\to (X^{n+1},g_+)$ be
	a smooth proper mapping
	admitting polyhomogeneous expansion \eqref{eq:polyhomogeneous-expansion-of-map}
	between conformally compact Einstein manifolds.
	We say that $u$ has \emph{even} expansion if it is of the form
	\begin{equation}
		u^0=\sum_{\substack{1\leqq k\leqq q+1 \\ \text{$k$ odd}}}
		s^ku_k^0+o(s^{q+1}),\qquad
		u^i=\varphi^i+\sum_{\substack{1\leqq k\leqq q+1 \\ \text{$k$ even}}}
		s^ku_k^i+o(s^{q+1}).
	\end{equation}
	Moreover, we say $u$ is \emph{regular} if $\varphi$ has nowhere vanishing differential and
	$u_1^0$ is nowhere vanishing.
\end{dfn}

Of course, the above definition is motivated by
the following observation on harmonic maps between conformally compact Einstein manifolds.
Formula \eqref{eq:harmonic-map-equation-principal-part} implies that the regularity condition is satisfied
if $d\varphi$ is nowhere vanishing and $\tau(u)=o(1)$.

To keep the description of our energy renormalization less complicated,
from now on, let us focus on the case where $q$ is odd.
(This assumption must be introduced at some point anyway, because only for this case
the renormalized energy $\Eren$ can be defined invariantly.)
We moreover assume that $n>q$, and that
$h_s$ has \emph{strongly even} expansion in the sense that
\begin{equation}
	h_s\sim h+\sum_{\nu=1}^\infty s^{2\nu}h_{2\nu},
\end{equation}
which in fact follows if we assume the vanishing of the formally undetermined $q$-th order term
in the middle line of \eqref{eq:Einstein-metric-even-expansion}.
It follows from Lemma \ref{lem:Graham-lemma-for-changing-GL-normalization}
that the vanishing of the formally undetermined term
does not depend on any particular choice of the Graham--Lee normalization.
If $h_s$ has a strongly even expansion, then all $h_{2\nu}$ are formally determined by $h$.

In the following proposition, we call a coefficient of $u$ \emph{formally determined} when
it has a local expression in terms of $\varphi$, $h$, and $g$.

\begin{prop}
	\label{prop:even-polyhomogeneous-expansion-of-harmonic-maps}
	Let $(Y^{q+1},h_+)$ and $(X^{n+1},g_+)$ be two conformally compact Einstein manifolds,
	where $q\geqq 1$ is odd and $n>q$, and assume moreover that $h_s$ has strongly even expansion.
	Let $u\colon (Y^{q+1},h_+)\to (X^{n+1},g_+)$ be a proper map admitting polyhomogeneous expansion
	such that $\varphi\colon N\to M$ has nowhere vanishing differential and $\tau(u)=o(s^{q+1})$,
	i.e., each component of $\tau(u)$ satisfies $\tau(u)^I=o(s^{q+2})$.
	Then $u$ has regular even expansion in the sense of Definition \ref{dfn:evenness-of-maps}.
	More precisely, the expansion of $u$ takes the form described below.

	(1) When $q=1$, we assume that $Y=\mathbb{H}^2$ and we use the upper-half plane coordinates.
	Then,
	\begin{subequations}
	\label{eq:harmonic-map-polyhomogeneous-expansion-for-q-one}
	\begin{align}
		\label{eq:harmonic-map-polyhomogeneous-expansion-for-q-one-normal}
		u^0&=su_1^0+s^3u_3^0+o(s^3),\\
		\label{eq:harmonic-map-polyhomogeneous-expansion-for-q-one-tangential}
		u^i&=\varphi^i+s^2u_2^i+s^3u_3^i+o(s^3),
	\end{align}
	\end{subequations}
	where $u_1^0$, $u_2^i$ are formally determined
	and $u_3^0$, $u_3^i$ are formally undetermined.
	\noeqref{eq:harmonic-map-polyhomogeneous-expansion-for-q-one-normal}
	\noeqref{eq:harmonic-map-polyhomogeneous-expansion-for-q-one-tangential}

	(2) When $q\geqq 3$,
	\begin{subequations}
	\label{eq:harmonic-map-polyhomogeneous-expansion-for-q-large}
	\begin{align}
		\label{eq:harmonic-map-polyhomogeneous-expansion-for-q-large-normal}
		u^0&=su_1^0+\dots+(\text{odd powers})+\dots+s^{q+2}u_{q+2}^0+s^{q+2}\log s\cdot v^0+o(s^{q+2}),\\
		\label{eq:harmonic-map-polyhomogeneous-expansion-for-q-large-tangential}
		u^i&=\varphi^i+s^2u_2^0+\dots+(\text{even powers})+\dots+
		s^{q+1}u_{q+1}^i+s^{q+2}u_{q+2}^i+o(s^{q+2}),
	\end{align}
	\end{subequations}
	where
	\begin{itemize}
		\item 
			For $k\leqq q+1$, $u_k^0$ and $u_k^i$ are always formally determined;
		\item
			$u_{q+2}^0$ is also formally determined when $n\geqq q+2$;
		\item
			$v^0$ is zero when $n\geqq q+2$;
		\item
			$u_{q+2}^i$ is always formally undetermined.
	\end{itemize}
	\noeqref{eq:harmonic-map-polyhomogeneous-expansion-for-q-large-normal}
	\noeqref{eq:harmonic-map-polyhomogeneous-expansion-for-q-large-tangential}
\end{prop}

\begin{proof}
	Part (1) is what we have checked quite explicitly in
	\S\S\ref{subsec:prelim-on-harmonic-maps}--\ref{subsec:third-order-coefficients}.
	Let us discuss part (2) here.

	As mentioned in \S\ref{subsec:prelim-on-harmonic-maps},
	to achieve $\tau(u)^I=o(s)$,
	there must be no terms of order $0<\nu<1$ and the expansion of
	$u^I$ must begin with \eqref{eq:first-normal-coefficient-of-harmonic-map}.
	Suppose that we have fixed the polyhomogeneous expansion of $u$ up to
	a certain order $\nu_0\geqq 1$ so that $\tau(u)^I=O(s^{\nu}(\log s)^L)$, where $\nu>\nu_0$.
	Then, if we set,
	\begin{equation}
		\tilde{u}^I=u^I+\sum_{l=0}^L s^\nu(\log s)^l u^I_{\nu,l}+o(s^\nu)
	\end{equation}
	a computation similar to \eqref{eq:change-of-tension-field-for-surfaces} shows that
	one can uniquely determine the coefficients $u_{\nu,l}^I$ for $0\leqq l\leqq L$ so that
	$\tau(\tilde{u})^I=o(s^\nu)$ is satisfied unless
	$\nu^2-(q+2)\nu-(q-1)=0$ or $\nu(\nu-q-2)=0$.
	Since the roots $\nu_\pm=\frac{1}{2}(q+2\pm\sqrt{q^2+8q})$ for the first equation
	satisfy $\nu_-<0<q+2<\nu_+$, one can see that the polyhomogeneous expansion of $u$ must take the form
	\begin{equation}
		\label{eq:primitive-expansion-of-harmonic-mapping-for-q-large}
		\begin{aligned}
			u^0&=\sum_{k=1}^{q+2}s^ku_k^0+s^{q+2}\log s\cdot v^0+o(s^{q+2}),\\
			u^i&=\varphi^i+\sum_{k=1}^{q+2}s^ku_k^i+s^{q+2}\log s\cdot v^i+o(s^{q+2}),
		\end{aligned}
	\end{equation}
	where the terms indicated by $o(s^{q+2})$ may contain terms of non-integer order.
	The appearance of a logarithmic term $s^{q+2}\log s\cdot v^0$ can happen only when $n=q+1$,
	in which case the term is caused by the logarithmic term in the expansion of $g_x$
	(see \eqref{eq:Einstein-metric-even-expansion}).

	Through the inductive procedure of determining the expansion,
	all the coefficients displayed in \eqref{eq:primitive-expansion-of-harmonic-mapping-for-q-large}
	except for $u_{q+2}^i$ can be expressed
	by a local formula in $\varphi$ and the coefficients in the expansions of $h_s$ and $g_x$.
	Since we are assuming that $h_s$ has strongly even expansion, all $h_{2\nu}$ are
	formally determined by $h$.
	On the other hand, $g_x$ contains non-local coefficients at order $n$ and higher,
	and formula \eqref{eq:formula-of-tension-field} implies the contribution of these coefficients
	in $\tau(u)^I$ is of order $O(s^{n+1})$.
	Therefore, the coefficients $u_k^I$ for $k\leqq q+1$ are always formally determined by $\varphi$, $h$ and $g$,
	and so are $u_{q+2}^0$ and $v^i$ when $n\geqq q+2$.

	The inductive procedure also implies that the truncated series
	\begin{equation}
		u^0_{(q+1)}=\sum_{k=1}^{q+1}s^ku_k^0,\qquad
		u^i_{(q+1)}=\varphi^i+\sum_{k=1}^{q+1}s^ku_k^i
	\end{equation}
	is the unique polynomial in $s$ of degree $q+1$ satisfying $\tau(u)^I=o(s^{q+1})$.
	If we formally replace the variables $s$ and $x$ with $-s$ and $-x$,
	then we obtain another solution to $\tau(u)^I=o(s^{q+1})$:
	\begin{equation}
		u^0_{(q+1),-}=\sum_{k=1}^{q+1}(-1)^{k+1}s^ku_k^0,\qquad
		u^i_{(q+1),-}=\varphi^i+\sum_{k=1}^{q+1}(-1)^ks^ku_k^i.
	\end{equation}
	Therefore, by the uniqueness, the expansion of $u^I_{(q+1)}$ has the evenness property.
	It then follows from \eqref{eq:formula-of-tension-field} that the $s^{q+2}$ term in $\tau(u_{(q+1)})^i$
	may occur from the terms in $\tensor{\Gamma}{^I_J_K}$ coming from the $n$-th order terms in $(g_x)_{ij}$,
	but a detailed investigation of $\tensor{\Gamma}{^I_J_K}$ (cf.\ \eqref{eq:Christoffel-symbols-for-g+})
	shows the vanishing of the $s^{q+2}$ term in $\tau(u_{(q+1)})^i$,
	which implies that $v^i=0$.
\end{proof}

\begin{rem}
	\label{rem:necessity-of-non-integer-powers}
	By pushing the above considerations further,
	we find that in general the polyhomogeneous expansions can only have $s^{k_1+k_2\nu_+}(\log s)^l$ terms,
	where $k_1$, $k_2$, $l\geqq 0$ are integers
	and $\nu_+=\frac{1}{2}(q+2+\sqrt{q^2+8q})$.
	A more precise form of the expansion can be found in Economakis \cite{Economakis-93-Thesis}.
\end{rem}

Let $u\colon (Y^{q+1},h_+)\to (X^{n+1},g_+)$ be a mapping between conformally compact Einstein manifolds
as in Proposition \ref{prop:even-polyhomogeneous-expansion-of-harmonic-maps}.
Its energy density $e(u)=\frac{1}{2}\abs{du}^2$, which is given by the formula
\begin{equation}
	e(u)=
	\frac{1}{2}\frac{s^2}{(u^0)^2}
	\Big((\partial_su^0)^2+h_s^{ab}(\partial_au^0)(\partial_bu^0)
	+(g_{u^0})_{ij}(\partial_su^i)(\partial_su^j)
	+h_s^{ab}(g_{u^0})_{ij}(\partial_au^i)(\partial_bu^j)\Big),
\end{equation}
has an expansion of the form
\begin{equation}
	e(u)=\sum_{\nu=0}^{(q-1)/2}s^{2\nu}e_{2\nu}+o(s^q),
\end{equation}
where all the coefficients $e_k$ are formally determined.
Since
\begin{equation}
	dV_{h_+}=\frac{ds\,dV_{h_s}}{s^{q+1}}
\end{equation}
and $dV_{h_s}$ expands as
\begin{equation}
	\label{eq:volume-density-expansion-for-domain}
	dV_{h_s}=dV_h\left(1+\sum_{\nu=1}^{(q-1)/2}s^{2\nu}v_{2\nu}+o(s^q)\right),
\end{equation}
we have
\begin{equation}
	\label{eq:energy-renormalization-expansion}
	\int_{\set{s\geqq\varepsilon}}e(u)\,dV_{h_+}=
	\sum_{\nu=0}^{(q-1)/2}\frac{c_{2\nu}}{\varepsilon^{q-2\nu}}+\Eren+o(1),
\end{equation}
and Lemma \ref{lem:Graham-lemma-for-changing-GL-normalization} implies that
the constant term $\Eren$ is independent of the choice of the Graham--Lee normalization of $(Y,h_+)$,
the proof of which is exactly the same as that of Theorem 3.1 in \cite{Graham-00}.
We call $\Eren$ the \emph{renormalized energy} of $u$.

\subsection{Critical points of the renormalized energy}

Obviously, the definition of the renormalized energy $\Eren$ generalizes to
any smooth proper map $u\colon(Y,h_+)\to (X,g_+)$ admitting regular even polyhomogeneous expansion.
Let us discuss variations of $u$ within the class of such maps.

The following lemma holds in the context of general Riemannian geometry.
For brevity, in this lemma we let $h$ and $g$ denote the metrics in $Y$ and $X$, instead of $h_+$ and $g_+$.

\begin{lem}
	\label{lem:infinitesimal-deformation-of-differential-of-maps}
	Let $u_t\colon(Y,h)\to(X,g)$ be a smooth family of smooth maps between Riemannian manifolds.
	For any $p\in Y$, we define
	the curve $\gamma_p\colon(-\varepsilon,\varepsilon)\to X$ by $\gamma_p(t)=u_t(p)$.
	Let $\Pi_t\colon T_{\gamma_p(t)}X\to T_{\gamma_p(0)}X$ be the parallel translation along $\gamma_p$,
	and for $v\in T_pY$, we write $w_t(v)=\Pi_t(du_t(v))\in T_{u_0(p)}X$.
	Then, if we write $\dot{u}=(du_t/dt)|_{t=0}$, it holds that
	\begin{equation}
		\left.\left(\frac{d}{dt}w_t\right)\right|_{t=0}=\nabla\dot{u},
	\end{equation}
	which is an equality between sections of $T^*Y\otimes u^*TX$.
\end{lem}

\begin{proof}
	Let $w(p,v,t,s)\in T_{\gamma_p(s)}X$ be the parallel translation of $du_t(v)$ along $\gamma_p$
	(note that $w_t(v)=w(p,v,t,0)$).
	To carry out necessary computations,
	we introduce local coordinates $(y^a)$ around $p\in Y$ and $(x^i)$ around $u(p)\in X$.
	Then we may regard $w(p,v,t,0)$ locally as the set of functions $w(y,v,t,0)^i$.
	It follows from the definition of $w$ that
	\begin{equation}
		\label{eq:parallel-translation-of-deformed-maps}
		w(y,v,t,t)^i=du_t(v)^i=v^a\partial_{y^a}u_t(y)^i
	\end{equation}
	and
	\begin{equation}
		\partial_sw(y,v,t,s)^i+\tensor{\Gamma}{^i_j_k}(u_s(y))\cdot\partial_su_s(y)^j\cdot w(y,v,t,s)^k=0.
	\end{equation}
	By differentiating \eqref{eq:parallel-translation-of-deformed-maps} in $t$, we obtain
	\begin{equation}
		\begin{split}
			v^a\partial_{y^a}\dot{u}(y)^i
			&=(\partial_tw)(y,v,0,0)^i+(\partial_sw)(y,v,0,0)^i \\
			&=\left.\left(\frac{d}{dt}w_t(v)^i\right)\right|_{t=0}
			-\tensor{\Gamma}{^i_j_k}(u(y))\cdot\dot{u}(y)^j\cdot du(v)^k.
		\end{split}
	\end{equation}
	Consequently,
	\begin{equation}
		\left.\left(\frac{d}{dt}w_t(v)^i\right)\right|_{t=0}
		=v^a\partial_{y^a}\dot{u}(y)^i+\tensor{\Gamma}{^i_j_k}(u(y))\cdot\dot{u}(y)^j\cdot du(v)^k
		=(\nabla_v\dot{u})^i,
	\end{equation}
	where the last equality follows because the Levi-Civita connection of $(X,g)$ is torsion-free.
\end{proof}

\begin{prop}
	\label{prop:characterization-of-criticality}
	Let $(Y^{q+1},h_+)$ and $(X^{n+1},g_+)$ be two conformally compact Einstein manifolds,
	where $q\geqq 1$ is odd and $n>q$, and assume that $h_s$ has strongly even expansion.
	Then a smooth proper map $u\colon (Y,h_+)\to (X,g_+)$ admitting regular even polyhomogeneous expansion
	is a critical map for the functional $\Eren$ among such mappings
	if and only if
	$u$ is harmonic and
	the formally undetermined coefficient
	$u^i_3$ in \eqref{eq:harmonic-map-polyhomogeneous-expansion-for-q-one} or
	$u^i_{q+2}$ in \eqref{eq:harmonic-map-polyhomogeneous-expansion-for-q-large} vanishes.
\end{prop}

\begin{proof}
	Suppose that $u$ is a critical map for $\Eren$.
	Considering compactly supported perturbations of $u$ shows that $u$ has to be harmonic
	just like in the case of mappings from closed manifolds.

	Therefore, we assume that $u$ is a harmonic map and
	we consider a smooth family $u_t$ of mappings admitting regular even polyhomogeneous expansion
	such that $u_0=u$. Let $\varphi_t\colon N\to M$ be the restriction of $u_t$ to the boundary.
	Then, since $\abs{du_t}^2$ equals $\abs{\Pi_t\circ du_t}^2$,
	it follows from Lemma \ref{lem:infinitesimal-deformation-of-differential-of-maps} that
	\begin{equation}
		\left.\left(\frac{d}{dt}e(u_t)\right)\right|_{t=0}
		=\left(\left.\frac{1}{2}\frac{d}{dt}\abs{du_t}^2\right)\right|_{t=0}
		=\braket{du,\nabla\dot{u}},
	\end{equation}
	and since $u$ is a harmonic map,
	\begin{equation}
		\frac{d}{dt}
		\left.\left(\int_{\set{s\geqq\varepsilon}}e(u)\,dV_{h_+}\right)\right|_{t=0}
		=\int_{\set{s\geqq\varepsilon}}\braket{du,\nabla\dot{u}}dV_{h_+}
		=-\int_{\set{s\geqq\varepsilon}}\nabla^*\braket{du,\dot{u}}dV_{h_+},
	\end{equation}
	where the last integrand is the divergence of the $1$-form $\braket{du,\dot{u}}$.
	By the divergence theorem, the last expression is equal to
	\begin{equation}
		\label{eq:variation-of-renormalized-energy}
		-\int_{s=\varepsilon}\braket{du(s\partial_s),\dot{u}}\frac{dV_{h_\varepsilon}}{\varepsilon^q}.
	\end{equation}
	In view of the expansion \eqref{eq:volume-density-expansion-for-domain} of $dV_{h_\varepsilon}$,
	the constant term in the asymptotic expansion of the above integral
	can be read off from the $s^{q-1}$, $s^{q-3}$, $\dotsc$ coefficients of $\braket{du(\partial_s),\dot{u}}$,
	i.e., $\tensor{(du)}{_0^I}\dot{u}^J(g_+)_{IJ}$.

	If $q=1$, since
	\begin{align}
		\tensor{(du)}{_0^0}&=u_1^0+3s^2u_3^0+s^2v^0+3s^2\log s\cdot v^0+\dotsb,\\
		\dot{u}^0&=s\dot{u}_1^0+s^3\dot{u}_3^0+s^3\log s\cdot\dot{v}^0+\dotsb,
	\end{align}
	there are no terms in $\tensor{(du)}{_0^0}\dot{u}^0(g_+)_{00}$ that concern. On the other hand,
	\begin{align}
		\tensor{(du)}{_0^i}&=2su_2^i+3s^2u_3^i+\dotsb,\\
		\dot{u}^j&=\dot{\varphi}^j+s^2\dot{u}_2^j+s^3\dot{u}_3+\dotsb,
	\end{align}
	and hence there is an $s^0$-term $3\abs{d\varphi}^{-2}\braket{u_3,\dot{\varphi}}_g$ in
	$\tensor{(du)}{_0^i}\dot{u}^j(u^0)^{-2}(g_{u^0})_{ij}$.
	Consequently, we have
	\begin{equation}
		\dot{\mathcal{E}}_\mathrm{ren}=-3\int_{N}\frac{\braket{u_3,\dot{\varphi}}_g}{\abs{d\varphi}^2}\,dV_h,
	\end{equation}
	and this vanishes for all smooth families $u_t$ if and only if $u_3=0$.
	The higher-dimensional case can be shown similarly.
\end{proof}

In view of Proposition \ref{prop:characterization-of-criticality},
when $u$ is given as a formal polyhomogeneous series, we say that $u$ is \emph{formally critical}
with respect to $\Eren$ when $\tau(u)^I=o(s^{q+2})$ and $u_{q+2}^i$ vanishes.
This definition extends further to the case in which $h_s$ and $g_x$ are given formally,
and furthermore, we can allow $h_+$ and $g_+$ to be pseudo-Riemannian.

It is natural to expect, and is indeed the case, that the formal criticality condition is
irrelevant to which Graham--Lee normalizations of $h_+$ and $g_+$ we use.
In fact, $\tau(u)^I=o(s^{q+2})$ is equivalent to $\tau(u)=o(s^{q+1})$,
which is obviously independent of the normalization,
so the point is the invariance of the condition $u_{q+2}^i=0$.
For simplicity, let us only give a detailed discussion here regarding changes of the normalization of $h_+$.
Recall that $h_+$ is assumed to be strongly even.
Under this assumption, the proof of Lemma \ref{lem:Graham-lemma-for-changing-GL-normalization}
shows that the change of coordinates $(s,t^i)\mapsto(\Hat{s},\Hat{t}^i)$
associated with different normalizations
satisfies $\Hat{s}=sb$ with $b=b(s,t^1,\dots,t^q)$ even in $s$ to infinite order.
Then, the discussion in \cite{Graham-Lee-91} following its Lemma 5.2 reveals that
\begin{equation}
	\frac{\partial s}{\partial\Hat{s}}
	=\frac{b+s\partial_s b}{(b+s\partial_s b)^2+s^2\abs{d_Nb}^2_{h_s}},\qquad
	\frac{\partial t^i}{\partial\Hat{s}}
	=\frac{s(\grad_{h_s}b)^i}{(b+s\partial_s b)^2+s^2\abs{d_Nb}^2_{h_s}},
\end{equation}
from which one can show that $(s,t^a)$ can be expressed in terms of $(\Hat{s},\Hat{t}^a)$ as
\begin{equation}
	s=\Hat{s}\psi^0(\Hat{s},\Hat{t}^1,\dots,\Hat{t}^q),\qquad
	t^a=\Hat{t}^a+\psi^a(\Hat{s},\Hat{t}^1,\dots,\Hat{t}^q),
\end{equation}
where $\psi^0$ and $\psi^a$ are even in $\Hat{s}$ to infinite order
and $\psi^0(0,\Hat{t}^1,\dots,\Hat{t}^q)>0$.
As a consequence, we can see that $\Hat{u}_{q+2}^i=0$ follows if $u_{q+2}^i=0$.
The proof of the invariance of $u_{q+2}^i=0$ for changes of the normalization of $g_+$ can be carried out similarly.

\begin{rem}
	\label{rem:localization}
	Since the above discussion regarding changes of the Graham--Lee normal forms is valid as well
	for local changes of normalizations,
	the formal criticality condition can be localized in the sense that it holds if and only if
	there exists an open covering $\set{\overline{U}_\lambda}$ of $\bdry_\infty Y$ in $\overline{Y}$
	and a choice of a Graham--Lee normalization in each $\overline{U}_\lambda$ such that,
	in each neighborhood,
	$\tau(u)^I=o(s^{q+2})$ and $u^i_{q+2}=0$ hold with respect to the chosen local normalization.
\end{rem}

Using the terminology we have just introduced,
we can say that a formal polyhomogeneous map $u\colon\mathbb{H}^2\to(X,g_+)$
as in Theorem \ref{thm:main-theorem-in-action} satisfying $\nabla du=o(s^2)$
is necessarily formally critical with respect to $\Eren$.

Finally, we remark that another notion of renormalized energy
is introduced by B\'erard \cite{Berard-13} in the context of conformally compact manifolds.
In his work, an energy renormalization scheme is given for
(non-proper) harmonic maps from conformally compact Einstein manifolds to closed manifolds.
While his work certainly belongs to the same cultural sphere as ours, his renormalized energy is
different from what we have described.
In particular, B\'erard's work provides conformally invariant objects on the boundary at infinity of the domain
of a harmonic map,
while we have been mainly interested in conformal geometry of the boundary at infinity of the target
in this paper.

\bibliography{myrefs}

\end{document}